\documentclass[11pt]{article}
\usepackage[a4paper,margin=30mm]{geometry}
\usepackage[T1]{fontenc}

\usepackage{amsmath,amsfonts,amssymb,amsthm,mathrsfs,esint}
\usepackage{enumerate}

\usepackage{url} 
\usepackage[colorlinks=true, linkcolor=blue, citecolor=blue, urlcolor=black]{hyperref}

\numberwithin{equation}{section}
\newtheorem{theorem}{Theorem}[section]
\newtheorem{lemma}[theorem]{Lemma}
\newtheorem{corollary}[theorem]{Corollary}
\newtheorem{proposition}[theorem]{Proposition}
\theoremstyle{definition}
\newtheorem{definition}[theorem]{Definition}
\newtheorem{example}[theorem]{Example}
\theoremstyle{remark}
\newtheorem*{remark}{Remark}

\allowdisplaybreaks

\title
{The endpoints $p_-=1$ and $p_+=\infty$ for the maximal operator 
on~$L^{p(\cdot)}$ over spaces of homogeneous type}

\author{Alina Shalukhina
\thanks{This work is funded by national funds through 
the FCT -- Funda\c{c}\~{a}o para a Ci\^encia e a Tecnologia, I.P., 
under the scope of the projects \texttt{UID/00297/2025} and 
\texttt{UID/PRR/00297/2025} (Center for Mathematics and Applications -- NOVA Math) 
and under the scope of the PhD scholarship \texttt{UI/BD/154284/2022.}} 
\\[.5ex]
\small Center for Mathematics and Applications and Department of Mathematics, \\
\small NOVA School of Science and Technology, NOVA University Lisbon \\
\small Quinta da Torre, Caparica, 2829--516, Portugal \\
\small \texttt{a.shalukhina@campus.fct.unl.pt}
}

\begin{document}
\date{}
\maketitle

\begin{abstract}
\noindent
We investigate the boundedness of the Hardy--Littlewood maximal operator~$M$ on 
variable Lebesgue spaces~$L^{p(\cdot)}(\Omega,d,\mu)$ over spaces of homogeneous type, 
focusing on the endpoints $p_-=1$ and $p_+=\infty$. At the lower endpoint, we prove that
the boundedness of~$M$ forces $p_->1$ for all exponents if and only if the
space~$(\Omega,d,\mu)$ is non-atomic. At the upper endpoint, we generalize
the Muckenhoupt--Nekvinda sufficient condition of Adamadze, 
Diening, and Kopaliani. On bounded spaces of homogeneous type, whenever~$p_->1$, this 
yields a boundedness criterion based on the Muckenhoupt condition alone.
We also construct an unbounded space where the Muckenhoupt 
condition suffices for the maximal operator's boundedness.

\vspace{1em}
\\[0.5em]
\textbf{MSC~2020:} 42B25, 46E30, 30L99.
\end{abstract}

\section{Introduction}
Let~$(\Omega,d,\mu)$ be a space of homogeneous type, that is, a quasi-metric space with
a doubling measure, and let 
\[
Mf(x):=\sup_{B\ni x}\frac{1}{\mu(B)}\int_B |f(y)|\,d\mu(y),
\quad x\in\Omega,
\]
define the Hardy--Littlewood maximal operator~$M$. This work investigates the 
boundedness of~$M$ on variable Lebesgue spaces~$L^{p(\cdot)}(\Omega,d,\mu)$. The 
unifying theme of our results is the treatment of two critical endpoint regimes that 
are usually excluded from the theory: the lower endpoint~$p_-=1$ and the upper 
endpoint~$p_+=\infty$.

We first consider the lower endpoint. In the Euclidean setting, the boundedness
of~$M$ on~$L^{p(\cdot)}(\mathbb{R}^n)$ implies~$p_->1$. Cruz-Uribe,
Fiorenza, and Neugebauer~\cite[Theorem~1.7]{CFN03} established this necessity under 
the upper semi-continuity of the exponent, and 
Diening et al.~\cite[Theorem 6.3]{DHHMS09} later removed this regularity hypothesis. 
For spaces of homogeneous type, in joint work with Karlovych~\cite[Theorem~1]{KS25},
we proved the same necessity assuming the measure~$\mu$ to be Borel-semiregular and 
\emph{reverse doubling}. The latter condition, however, is 
not intrinsic to this endpoint phenomenon. Our first result identifies 
the exact geometric obstruction to the necessity of~$p_->1$: the presence 
of~\emph{atoms}, i.e., points $x\in\Omega$ with $\mu(\{x\})>0$. 
Theorems~\ref{th:KS} and~\ref{th:atom} prove the following characterization: 
\begin{theorem}[Lower endpoint]
Let~$(\Omega,d,\mu)$ be a space of homogeneous type with a Borel-semiregular 
measure~$\mu$. The implication 
\begin{center}
$M$ is bounded on~$L^{p(\cdot)}(\Omega,d,\mu)$
\quad$\implies$\quad 
$p_->1$
\end{center}
holds for every variable exponent~$p(\cdot)$ \emph{if and only if} the 
space~$(\Omega,d,\mu)$ is non-atomic. Equivalently, if the space contains an atom, 
there exists an exponent~$p(\cdot)$ with~$p_-=1$ for which $M$ is bounded.
\end{theorem}
Next, we turn to the upper endpoint. We generalize a recent Euclidean 
breakthrough by Adamadze, Diening, and Kopaliani~\cite[Theorem~1.1]{ADK26},
who proved that~$M$ is bounded on~$L^{p(\cdot)}(\mathbb{R}^n)$ whenever
$p_->1$ and the exponent~$p(\cdot)$ satisfies the variable Muckenhoupt and Nekvinda 
conditions, with no upper bound imposed on~$p_+$. This result extended the 
bounded-exponent sufficient condition~\cite[Theorem~4.52]{CUF13}, 
whose proof relied on Lerner’s local estimate~\cite[Theorem~1.2]{AL10}.

In joint work with Kopaliani and Tephnadze~\cite[Theorem~1.3]{KST26}, we extended 
Lerner's theorem to spaces of homogeneous type. This led to the open 
question~\cite[Problem~6.5]{KST26} of whether the Muckenhoupt and Nekvinda conditions 
admit an effective analogue in this setting and yield the boundedness of~$M$. 
Our second main result answers this question affirmatively, as Theorem~\ref{th:main-US} 
and Corollary~\ref{cor:criterion} generalize~\cite[Theorem~1.1]{ADK26} to spaces
of homogeneous type:
\begin{theorem}[Upper endpoint]\label{th:main-inf}
Let~$(\Omega,d,\mu)$ be an \emph{unbounded} space of homogeneous type with a 
Borel-semiregular measure~$\mu$. Suppose that a variable exponent~$p(\cdot)$
with $p_->1$ satisfies the following two conditions: 
\begin{enumerate}
\item[\textnormal{(i)}] \emph{(Muckenhoupt)}
$\displaystyle
\,\sup_{B} \tfrac{1}{\mu(B)}
\|\mathbf{1}_B\|_{p(\cdot)} \|\mathbf{1}_B\|_{p'(\cdot)}<\infty$;

\item[\textnormal{(ii)}] \emph{(Nekvinda)}
$\,1\in L^{s(\cdot)}(\Omega,\mu)$ with 
$\frac{1}{s(\cdot)}:=\big|\frac1{p(\cdot)}-\frac1{p_\infty}\big|$
for some~$p_\infty\in[1,\infty]$.
\end{enumerate}
Then~$M$ is bounded on~$L^{p(\cdot)}(\Omega,d,\mu)$. If the space~$(\Omega,d,\mu)$ is 
\emph{bounded}, the Muckenhoupt condition alone is necessary and sufficient for
the boundedness of~$M$ whenever~$p_->1$.
\end{theorem}
Our proof adapts the framework of~\cite{ADK26} but requires 
several non-Euclidean ingredients. The adjacent dyadic grids of
Hyt\"onen and Kairema~\cite{HK12} reduce the classical maximal operator to finitely 
many dyadic ones, while a dyadic Calder\'on--Zygmund decomposition supplies 
the power-decay characterization for local~$A_\infty$ weights defined 
directly on the dyadic tree. The approximation of an unbounded exponent by bounded 
ones~\cite[Lemma~4.4]{ADK26} is purely measure-theoretic and survives unchanged once 
the Muckenhoupt and Nekvinda characteristics are available.

We emphasize a few technical improvements. First, our refined 
Lemma~\ref{le:power-decay} 
justifies a scaled-weight transition left unexplained in~\cite[Proposition~5.4]{ADK26};
see \emph{Remark} following the lemma. Second, Corollary~\ref{cor:aux-US} explicitly 
shows that the bounded-exponent norm estimate depends solely on the \emph{upper bounds} 
of the exponent's characteristics. Since every exponent provides these 
uniform bounds for its approximating sequence, this corollary enables
passing to the limit for~$p_+=\infty$. We also streamline the proof of 
Theorem~\ref{th:aux-US} by isolating the elementary Nekvinda embeddings into 
Lemma~\ref{le:N-emb}, and we carefully separate the cases of bounded and unbounded 
spaces to handle their different geometries.

Finally, we resolve~\cite[Problem~6.4]{KST26}, which asks 
whether the Muckenhoupt condition is strictly insufficient on every
unbounded space of homogeneous type. The answer is negative. \emph{We construct an 
unbounded atomic space~$(\Omega,d,\mu)$ where the Muckenhoupt condition alone ensures 
the boundedness of~$M$ on~$L^{p(\cdot)}(\Omega,d,\mu)$} for any exponent 
with $1<p_-\le p_+<\infty$. This example excludes the endpoints but 
contextualizes Theorem~\ref{th:main-inf}. It also serves as an application of our 
recent characterization of the maximal operator's boundedness via uniform estimates for 
dyadic averaging operators~\cite[Theorem~1.2]{AS26}.

The paper is organized as follows. Section~\ref{sec:var-Lp} collects the 
variable-exponent preliminaries and introduces the Nekvinda condition. 
Section~\ref{sec:SHT} recalls spaces of homogeneous type, Hyt\"onen--Kairema
dyadic cubes, and the standard and dyadic Muckenhoupt conditions. 
Section~\ref{sec:p-=1} proves the atomic classification at the lower endpoint. 
Sections~\ref{sec:Jensen} and~\ref{sec:p+<inf} develop the machinery for the main 
boundedness theorems, which are proved in Section~\ref{sec:main-ADK}. 
Section~\ref{sec:US-Muck-suff} presents an unbounded atomic example where the 
Muckenhoupt condition suffices.

\section{Variable Exponents and the Nekvinda Condition}
\label{sec:var-Lp}
Let $(\Omega,\mu)$ be a $\sigma$-finite, complete measure space, and denote 
by~$L^0(\Omega,\mu)$ the space of all $\mu$-measurable functions on~$\Omega$.
Let~$\mathcal{P}(\Omega,\mu)\subseteq L^0(\Omega,\mu)$ be the set of 
all~\emph{variable exponents}
\[
p(\cdot)\,:\,\Omega\to[1,\infty].
\]
We characterize each~$p(\cdot)\in\mathcal{P}(\Omega,\mu)$ by its essential bounds
\[
p_-:=\operatornamewithlimits{ess\,inf}_{x\in\Omega}p(x)
\quad\text{and}\quad
p_+:=\operatornamewithlimits{ess\,sup}_{x\in\Omega}p(x).
\]
The \emph{conjugate exponent}~$p'(\cdot)\in\mathcal{P}(\Omega,\mu)$
is defined pointwise by the relation $1/p(\cdot)+1/p'(\cdot)=1$,
where it is understood that $1/\infty=0$.

Every exponent~$p(\cdot)\in\mathcal{P}(\Omega,\mu)$ 
induces the \emph{semimodulars}~$\varrho_{p(\cdot)}$ and~$\tilde{\varrho}_{p(\cdot)}$,
defined for all~$f\in L^0(\Omega,\mu)$ by
\[
\varrho_{p(\cdot)}(f):=\int_\Omega |f(x)|^{p(x)}\,d\mu(x)
\quad\textnormal{and}\quad
\tilde{\varrho}_{p(\cdot)}(f):=\int_\Omega \frac{|f(x)|^{p(x)}}{p(x)}\,d\mu(x).
\]
To evaluate the integrands when $p(x) = \infty$, we adopt the
convention that $t^\infty=t^\infty/\infty=0$ for~$t\in[0,1]$, and
$t^\infty=t^\infty/\infty=\infty$ for~$t>1$. Consequently, the semimodulars 
are \emph{semi-} in the sense that $\varrho_{p(\cdot)}(f)=0$ or 
$\tilde{\varrho}_{p(\cdot)}(f)=0$ need not imply $f=0$.
\begin{definition}
Given an exponent $p(\cdot)\in\mathcal{P}(\Omega,\mu)$, the 
\emph{variable Lebesgue space~$L^{p(\cdot)}(\Omega,\mu)$} consists of all
functions $f\in L^0(\Omega,\mu)$ such that
\[
\|f\|_{p(\cdot)}:=
\inf\big\{\lambda>0\,:\,\varrho_{p(\cdot)}(f/\lambda)\le1\big\}<\infty.
\]
\end{definition}
Equipped with this norm, $L^{p(\cdot)}(\Omega,\mu)$ is a Banach function 
space~\cite[Theorem~3.2.13]{DHHR11} that recovers the classical Lebesgue space when 
$p(\cdot)$ is constant. Alternatively, one can define the space 
via the functional
\[
\|f\|_{\tilde{\varrho}_{p(\cdot)}}:=
\inf\big\{\lambda>0\,:\,\tilde{\varrho}_{p(\cdot)}(f/\lambda)\le1\big\},
\]
which yields an equivalent norm~\cite[p.~74]{DHHR11} satisfying
\begin{equation}\label{eq:norms}
\|f\|_{\tilde{\varrho}_{p(\cdot)}}
\le\|f\|_{p(\cdot)}\le2\,\|f\|_{\tilde{\varrho}_{p(\cdot)}}.
\end{equation}

The following theorem, which collects fundamental properties of variable
Lebesgue spaces, will serve as a toolbox for our subsequent proofs.
\begin{theorem}[Toolbox for variable exponents]\label{th:toolbox}
Let $p(\cdot), q(\cdot),s(\cdot)\in\mathcal{P}(\Omega,\mu)$. The following properties 
hold:
\begin{enumerate}

\item[\textnormal{(a)}] \emph{(unit ball property)} 
For all $f\in L^{p(\cdot)}(\Omega,\mu)$,
\[
\|f\|_{p(\cdot)}\le1
\quad\text{if and only if}\quad
\varrho_{p(\cdot)}(f)\le 1.
\]
The same equivalence holds for the semimodular $\tilde{\varrho}_{p(\cdot)}$ and its 
associated norm.

\item[\textnormal{(b)}] \emph{(unit sphere property)} If $p_+<\infty$, then
$\|f\|_{p(\cdot)}=1$ is equivalent to $\varrho_{p(\cdot)}(f)=1$.

\item[\textnormal{(c)}] \emph{(semimodular-to-norm bound)} If 
$\varrho_{p(\cdot)}(f)\le ab$ for some $a\ge1$ and $b\in(0,1]$, 
then $\|f\|_{p(\cdot)}\le ab^{\frac1{p_+}}$.

\item[\textnormal{(d)}] \emph{(H\"older's inequality)} Let 
$1/s(\cdot)=1/p(\cdot)+1/q(\cdot)$.
Then for all $f\in L^{p(\cdot)}(\Omega,\mu)$ and $g\in L^{q(\cdot)}(\Omega,\mu)$,
\[
\|fg\|_{s(\cdot)}\le2\,\|f\|_{p(\cdot)}\|g\|_{q(\cdot)}.
\]

\item[\textnormal{(e)}] \emph{(norm conjugate formula)}
For all $f\in L^0(\Omega,\mu)$,
\[
\|f\|_{p(\cdot)}\le2\sup_{\|g\|_{p'(\cdot)}\le1}
\int_\Omega |f(x)\,g(x)|\,d\mu(x),
\]
where the supremum remains unchanged if restricted to bounded functions $g$ 
supported on a set of finite measure.

\item[\textnormal{(f)}] \emph{(squeezed embedding)} If 
$s(\cdot)\le p(\cdot)\le q(\cdot)$, then for all 
$f\in L^{s(\cdot)}(\Omega,\mu)\cap L^{q(\cdot)}(\Omega,\mu)$,
\[
\|f\|_{p(\cdot)}\le 2\,\max\{\|f\|_{s(\cdot)},\|f\|_{q(\cdot)}\}.
\]

\item[\textnormal{(g)}] \emph{(Fatou property)} For $f_k,f\in L^0(\Omega,\mu)$, 
if $0\le f_k\nearrow f$ $\mu$-almost everywhere, then
\[
\|f\|_{p(\cdot)}=\lim_{k\to\infty}\|f_k\|_{p(\cdot)}
\quad\text{and}\quad
\|f\|_{\tilde{\varrho}_{p(\cdot)}}=\lim_{k\to\infty}\|f_k\|_{\tilde{\varrho}_{p(\cdot)}}.
\]

\item[\textnormal{(h)}] \emph{(Riesz--Fischer property)} For any nonnegative sequence 
$\{f_k\}_{k=0}^\infty\subseteq L^0(\Omega,\mu)$, 
\[
\bigg\|\sum_{k=0}^\infty f_k\bigg\|_{p(\cdot)}\le \sum_{k=0}^\infty \|f_k\|_{p(\cdot)}.
\]

\item[\textnormal{(i)}] \emph{(power-type rescaling)} For all $s\ge1$ and 
$f\in L^{sp(\cdot)}(\Omega,\mu)$, 
\[
\|\,|f|^s\|_{p(\cdot)}=\|f\|_{sp(\cdot)}^s.
\]

\item[\textnormal{(j)}] \emph{(interpolation estimate)} 
Let $1/p(\cdot)=\theta/s(\cdot)+(1-\theta)/q(\cdot)$ for some $\theta\in[0,1]$. 
Then for all $f\in L^{s(\cdot)}(\Omega,\mu)\cap L^{q(\cdot)}(\Omega,\mu)$,
\[
\|f\|_{p(\cdot)}\le2\,\|f\|_{s(\cdot)}^\theta\|f\|_{q(\cdot)}^{1-\theta}.
\]

\item[\textnormal{(k)}] \emph{(lower semi-continuity)} Let a sequence of 
exponents $p_n(\cdot)\in\mathcal{P}(\Omega,\mu)$ be such that $p_n(\cdot)\to p(\cdot)$
$\mu$-almost everywhere. Then for all $f\in L^0(\Omega,\mu)$,
\[
\|f\|_{p(\cdot)}\le\liminf_{n\to\infty}\|f\|_{p_n(\cdot)}
\quad\text{and}\quad
\|f\|_{\tilde{\varrho}_{p(\cdot)}}\le\liminf_{n\to\infty}
\|f\|_{\tilde{\varrho}_{p_n(\cdot)}}.
\]

\item[\textnormal{(l)}] \emph{(semimodular convergence)} Let 
$p_n(\cdot)\in\mathcal{P}(\Omega,\mu)$ be squeezed as
$s(\cdot)\le p_n(\cdot)\le q(\cdot)$, and let $p_n(\cdot)\to p(\cdot)$
$\mu$-almost everywhere. Then for all $f\in L^0(\Omega,\mu)$ satisfying
$\tilde{\varrho}_{s(\cdot)}(f)+\tilde{\varrho}_{q(\cdot)}(f)<\infty$,
\[
\lim_{n\to\infty}
\tilde{\varrho}_{p_n(\cdot)}(f)
=\tilde{\varrho}_{p(\cdot)}(f).
\]

\item[\textnormal{(m)}] \emph{(convexity scaling)} For any $\lambda\in(0,1)$
and $f\in L^0(\Omega,\mu)$,
\[
\tilde{\varrho}_{p(\cdot)}(\lambda f)\le\lambda\,\tilde{\varrho}_{p(\cdot)}(f).
\]

\end{enumerate}
\end{theorem}
\begin{proof}
The listed properties of variable Lebesgue spaces are well known, and many of them can 
be found in the monograph~\cite{DHHR11}: 
(a) and~(b) correspond to Lemma~2.1.14; 
(d)~is part of Lemma~3.2.20; 
(e)~is Corollary~3.2.14;
(f)~is part of Theorem~3.3.11; 
(g)~follows from Theorem~2.3.17;
(k)~is Corollary~3.5.4;
(l)~is Lemma~3.5.5;
(m)~is implicitly stated in Section~3.2.

Let us comment on the rest of the statements. Properties~(c) and~(i) follow immediately 
from the definitions. The Riesz--Fischer property~(h) is deduced from the Fatou 
property~(g) exactly as in~\cite[Chapter~1, Theorem~1.6]{BS88}. Interpolation 
estimate~(j) follows by the factorization $|f|=|f|^\theta\,|f|^{1-\theta}$ and 
subsequent application of H\"older's inequality~(d) and rescaling~(i).
\end{proof}

When $\mu(\Omega)=\infty$, it is useful to control the behavior of variable exponents
analogously to how they are restricted \emph{at infinity} in Euclidean spaces. This is
achieved via the Nekvinda condition, originally introduced 
in~\cite[Definition~2.13]{AN04} as an explicit integral decay requirement
on~$\mathbb{R}^n$, which naturally translates to general measure spaces in the 
functional form noted in~\cite[Remark~4.2.8]{DHHR11}.
\begin{definition}
Let $\mu(\Omega)=\infty$. An exponent $p(\cdot)\in\mathcal{P}(\Omega,\mu)$ satisfies 
the \emph{Nekvinda condition}, denoted $p(\cdot)\in\mathcal{N}(\Omega,\mu)$, if there 
exists $p_\infty\in[1,\infty]$ such that $1\in L^{s(\cdot)}(\Omega,\mu)$, where 
\begin{equation}\label{cond:N}
\frac1{s(x)}:=\bigg|\frac1{p(x)}-\frac1{p_\infty}\bigg|,
\quad x\in\Omega.
\end{equation}
The corresponding characteristic is $[p]_\mathcal{N}:=\|1\|_{s(\cdot)}$.
\end{definition}
Note that $[p]_\mathcal{N}\ge1$, since $\varrho_{s(\cdot)}(t)\ge\mu(\Omega)=\infty$
for all $t>1$. Moreover, $p_\infty\in[p_-,p_+]$. If $p(\cdot)$ satisfies the Nekvinda 
condition with parameter~$p_\infty$, then $p'(\cdot)$ also satisfies it with 
parameter~$p'_\infty$, and $[p]_\mathcal{N}=[p']_\mathcal{N}$. 
The condition $\mathcal{N}$ also implies the following norm bounds.
\begin{lemma}[Elementary Nekvinda embeddings]\label{le:N-emb}
Let $q(\cdot)\in\mathcal{N}(\Omega,\mu)$ with parameter~$q_\infty$. The following hold:
\begin{enumerate}
\item[\textnormal{(a)}] For any $f\in L^{q_\infty}(\Omega,\mu)\cap L^{q_+}(\Omega,\mu)$,
\[
\|f\|_{q(\cdot)}\le 4[q]_\mathcal{N}\,\max\{\|f\|_{q_\infty},\|f\|_{q_+}\}.
\]

\item[\textnormal{(b)}] If $0\le g\le1$ and $\|g\|_{q(\cdot)}\le\frac12$, then 
\[
\max\{\|g\|_{q_\infty},\|g\|_{q_+}\}\le 3[q]_\mathcal{N}.
\]
\end{enumerate}
\end{lemma}
\begin{proof}
Let $s(\cdot)$ be the exponent~\eqref{cond:N} from the Nekvinda condition associated 
with $q(\cdot)$. Partition~$\Omega$ into $\Omega_-:=\{q(\cdot)<q_\infty\}$ and
$\Omega_+:=\{\,q(\cdot)\ge q_\infty\}$. We first prove~(a). On~$\Omega_-$, we have 
$1/q(\cdot)=1/q_\infty+1/s(\cdot)$, and H\"older's
inequality yields
\[
\|f\mathbf{1}_{\Omega_-}\|_{q(\cdot)}
\le 2\,\|f\|_{q_\infty}\|\mathbf{1}_{\Omega_-}\|_{s(\cdot)}
\le 2[q]_\mathcal{N}\,\|f\|_{q_\infty}.
\]
On $\Omega_+$, the exponent is squeezed as $q_\infty\le q(\cdot)\le q_+$. 
Theorem~\ref{th:toolbox}(f) thus implies
\[
\|f\mathbf{1}_{\Omega_+}\|_{q(\cdot)}\le 2\,\max\{\|f\|_{q_\infty},\|f\|_{q_+}\}.
\]
Summing these estimates and using $[q]_\mathcal{N}\ge 1$ establishes~(a).

For~(b), applying H\"older's inequality on~$\Omega_+$ gives
\[
\|g\mathbf{1}_{\Omega_+}\|_{q_\infty}\le 2\,\|\mathbf{1}_{\Omega_+}\|_{s(\cdot)}
\|g\|_{q(\cdot)}\le[q]_\mathcal{N}.
\]
On $\Omega_-$, the squeeze $q(\cdot)<q_\infty\le\infty$ combined with the bounds
$\|g\|_{q(\cdot)}\le\frac12$ and $\|g\|_\infty\le1$ ensures that
$\|g\mathbf{1}_{\Omega_-}\|_{q_\infty}\le2$. Summing these parts gives 
$\|g\|_{q_\infty}\le [q]_\mathcal{N}+2\le 3[q]_\mathcal{N}$. 
Finally, the global bound $q(\cdot)\le q_+\le\infty$ yields 
$\|g\|_{q_+}\le 2$, completing the proof.
\end{proof}

\section{Dyadic Grids and the Muckenhoupt Condition}
\label{sec:SHT}
Spaces of homogeneous type generalize the Euclidean setting to quasi-metric spaces 
with a doubling measure. This framework unifies diverse geometric 
structures, including Ahlfors-regular spaces, homogeneous groups, and compact 
Riemannian manifolds. For examples and basic properties, see the foundational works 
of Coifman and Weiss~\cite{CW71,CW77} and the monographs~\cite{AM15,GGKK98}.
\begin{definition}
A \emph{space of homogeneous type} is a triple $(\Omega,d,\mu)$ equipped with structural 
constants $\kappa\ge1$ and $C_\mu>1$ such that:
\begin{enumerate}
\item[\textnormal{(i)}] $d$ is a \emph{quasi-metric} on~$\Omega$ satisfying 
\[
d(x,z)\le\kappa\,(d(x,y) + d(y,z)) 
\quad\textnormal{for all }x,y,z\in\Omega;
\]
\item[\textnormal{(ii)}] $\mu$ is a Borel measure on~$\Omega$ such that balls
$B(x,r):=\{y\in\Omega\,:\,d(x,y)<r\}$ are $\mu$-measurable and satisfy the
\emph{doubling condition} 
\begin{equation}\label{cond:doubling}
0<\mu(B(x,2r)) \le C_\mu\, \mu(B(x,r))<\infty.
\end{equation}
\end{enumerate}
\end{definition}
The space~$(\Omega,d,\mu)$ is 
\emph{bounded if and only if $\mu(\Omega)<\infty$}~\cite[Lemma~1.9]{BC96}. 
We rely on this fundamental equivalence to separate the bounded and unbounded cases 
in our subsequent proofs. For any set~$E\subseteq\Omega$ with $\mu(E)<\infty$, we 
denote the average of~$f\in L^0(\Omega,\mu)$ by
\[
|f|_E:=\fint_E |f(y)|\,d\mu(y)
=\frac1{\mu(E)} \int_E |f(y)|\,d\mu(y).
\]

The \emph{Hardy--Littlewood maximal operator~$M$} is defined on~$L^0(\Omega,d,\mu)$ by
\[
Mf(x):=\sup_{B\ni x}|f|_B,
\quad x\in\Omega,
\]
where the supremum is taken over all balls $B\subseteq\Omega$ containing~$x$.
It is well known that~$M$ is bounded on~$L^p(\Omega,d,\mu)$ for all 
$1<p\le\infty$~\cite[Theorem~3.7 and Comment~3.8]{AM15}. In the variable setting, 
the following Muckenhoupt condition is \emph{necessary} for the boundedness of~$M$ 
on~$L^{p(\cdot)}(\Omega,d,\mu)$:
\begin{definition}
An exponent $p(\cdot)\in\mathcal{P}(\Omega,\mu)$ satisfies the 
\emph{Muckenhoupt condition}, denoted 
$p(\cdot)\in\mathcal{A}_\textnormal{loc}(\Omega,d,\mu)$, if 
\[
[p]_\mathcal{A}:=\sup_{B\subseteq\Omega} \frac1{\mu(B)} 
\|\mathbf{1}_B\|_{p(\cdot)} \|\mathbf{1}_B\|_{p'(\cdot)}<\infty,
\]
where the supremum is taken over all balls in $(\Omega,d,\mu)$.
\end{definition}
Note that $[p]_\mathcal{A}\ge\frac12$ by H\"older's inequality. Moreover, if 
$p(\cdot)$ satisfies the Muckenhoupt condition, then $p'(\cdot)$ also 
satisfies it with $[p']_\mathcal{A}=[p]_\mathcal{A}$.
\begin{lemma}\label{le:necess-Muck}
For any $p(\cdot)\in\mathcal{P}(\Omega,\mu)$, the boundedness of~$M$ 
on~$L^{p(\cdot)}(\Omega,d,\mu)$ implies
$p(\cdot)\in\mathcal{A}_\textnormal{loc}(\Omega,d,\mu)$.
\end{lemma}
\begin{proof}
Fix a ball~$B$. By the norm conjugate formula from Theorem~\ref{th:toolbox}(e),
\begin{align*}
\frac1{\mu(B)} \|\mathbf{1}_B\|_{p(\cdot)} \|\mathbf{1}_B\|_{p'(\cdot)}
&\le
2\, \|\mathbf{1}_B\|_{p(\cdot)}\, \sup_{\|g\|_{p(\cdot)}\le1}
\fint_B |g(x)|\,d\mu(x)
\\
&=2\,\sup_{\|g\|_{p(\cdot)}\le1} \|\,|g|_B\,\mathbf{1}_B\|_{p(\cdot)}
\\
&\le 2\,\|M\|_{L^{p(\cdot)}\to L^{p(\cdot)}},
\end{align*}
which is a finite, uniform bound for all balls~$B$. Thus $[p]_\mathcal{A}<\infty$.
\end{proof}

Further study of the maximal operator on $L^{p(\cdot)}(\Omega,d,\mu)$
is facilitated by the dyadic structure on spaces of homogeneous type. The following 
construction of dyadic grids is due to Hyt\"onen and Kairema~\cite[Theorem~2.2]{HK12}, 
stated here in the simplified form of~\cite[Theorem~2.1]{ACUM15}.
\begin{theorem}[Dyadic cubes]\label{th:HK}
Let $(\Omega,d,\mu)$ be a space of homogeneous type. There exist constants 
$C_\mathrm{D},R\ge1$ and  $0<\delta<1$, depending on the structural constants of the space; a family 
of sets 
\[
\mathscr{D}:=\bigcup_{k\in \mathbb{Z}}\mathscr{D}_{k},
\]
called \emph{a dyadic grid on $\Omega$}, whose elements are referred to as 
\emph{dyadic cubes}; and a collection of centers $\{x_Q\}_{Q\in\mathscr{D}}$ such that:
\begin{enumerate}
\item[\textnormal{(a)}] for each generation $k\in\mathbb{Z}$, the collection 
$\mathscr{D}_k$ is a partition of $\Omega$;

\item[\textnormal{(b)}] two dyadic cubes are either disjoint or one contains the other;

\item[\textnormal{(c)}] every cube $Q\in\mathscr{D}_k$ has a unique \emph{parent} 
in~$\mathscr{D}_{k-1}$ containing it, and at least one \emph{child} 
in~$\mathscr{D}_{k+1}$ contained within it;

\item[\textnormal{(d)}] if $Q'\in\mathscr{D}_{k+1}$ is a child of $Q\in\mathscr{D}_k$, 
then
\[
\mu(Q)\le C_\mathrm{D}\,\mu(Q'),
\]
where $C_\mathrm{D}$ is called the \emph{dyadic doubling constant};

\item[\textnormal{(e)}] for all $Q\in\mathscr{D}_k$,
\[
B(x_Q,\delta^{k})\subseteq Q\subseteq B(x_Q,R\,\delta^{k}).
\]
\end{enumerate}
\end{theorem}
Each dyadic grid~$\mathscr{D}$ induces a
\emph{dyadic maximal operator}~$M^\mathscr{D}$, 
defined for all~$f\in L^0(\Omega,d,\mu)$ by 
\[
M^\mathscr{D}f(x):=\sup_{\substack{Q\ni x,\\Q\in\mathscr{D}}} |f|_Q,
\quad x\in\Omega,
\]
where the supremum is taken over all cubes $Q\in\mathscr{D}$ containing~$x$.
Analogously, the boundedness of~$M^\mathscr{D}$ on~$L^{p(\cdot)}(\Omega,d,\mu)$
necessitates a \emph{dyadic} Muckenhoupt condition:
\begin{definition}
Given a dyadic grid $\mathscr{D}$ on $\Omega$, an exponent 
$p(\cdot)\in\mathcal{P}(\Omega,\mu)$ satisfies the 
\emph{dyadic Muckenhoupt condition}, denoted 
$p(\cdot)\in\mathcal{A}^\mathscr{D}_\textnormal{loc}(\Omega,d,\mu)$, if
\[
[p]_\mathcal{A}^\mathscr{D}:=\sup_{Q\in\mathscr{D}} \frac1{\mu(Q)}
\|\mathbf{1}_Q\|_{p(\cdot)}\|\mathbf{1}_Q\|_{p'(\cdot)}<\infty.
\]
\end{definition}
\begin{lemma}[Two Muckenhoupt conditions]\label{le:2-Mucken}
If $p(\cdot)\in\mathcal{A}_\textnormal{loc}(\Omega,d,\mu)$, then 
for every dyadic grid~$\mathscr{D}$ on~$\Omega$, one has
$p(\cdot)\in\mathcal{A}^\mathscr{D}_\textnormal{loc}(\Omega,d,\mu)$.
Specifically,
\[
[p]_\mathcal{A}^\mathscr{D}\le C_\Omega\,[p]_\mathcal{A},
\]
where $C_\Omega>1$ depends only on the structural constants of~$(\Omega,d,\mu)$.
\end{lemma}
\begin{proof}
For $Q\in\mathscr{D}_k$, let $B_-:=B(x_Q,\delta^k)$ and $B_+:=B(x_Q,R\,\delta^k)$. 
By Theorem~\ref{th:HK}(e) and the doubling condition~\eqref{cond:doubling}, there
exists a constant $C_\Omega>1$, depending only on~$C_\mu$ and~$R$, such that 
$\mu(B_+)\le C_\Omega\,\mu(B_-)\le C_\Omega\,\mu(Q)$. Then 
\[
\frac1{\mu(Q)}\|\mathbf{1}_Q\|_{p(\cdot)}\|\mathbf{1}_Q\|_{p'(\cdot)}
\le\frac{C_\Omega}{\mu(B_+)}\|\mathbf{1}_{B_+}
\|_{p(\cdot)}\|\mathbf{1}_{B_+}\|_{p'(\cdot)}
\le C_\Omega\,[p]_\mathcal{A},
\]
which yields the desired inequality.
\end{proof}
The level sets of the dyadic maximal function~$M^\mathscr{D}f$ admit the following 
Calder\'on--Zygmund decomposition~\cite[Theorem~9 and Corollary~10]{AK19}.
\begin{lemma}[Calder\'on--Zygmund decomposition]\label{le:CZ}
Let $(\Omega,d,\mu)$ be a space of homogeneous type, $\mathscr{D}$ be a dyadic 
grid on~$\Omega$ with dyadic doubling constant $C_\mathrm{D}>1$, and 
$f\in L^1(\Omega,\mu)$. For any $a>2\,C_\mathrm{D}$ and 
$k\in\mathbb{Z}$ satisfying
\[
a^k>
\begin{cases}
0& \textnormal{if }\mu(\Omega)=\infty,
\\
|f|_\Omega & \textnormal{if }\mu(\Omega)<\infty,
\end{cases}
\]
there exist \emph{pairwise disjoint} maximal cubes $Q_j^k\in\mathscr{D}$ such that
$|f|_{Q_j^k}>a^k$ and
\[
\Omega_k:=\big\{x\in\Omega\,:\,M^{\mathscr{D}}f(x)>a^{k}\big\}
=\bigcup_{j}Q_j^k.
\]
Moreover, for all integers $\ell\ge0$ and all~$j$,
\begin{equation}\label{ineq:decay}
\frac{\mu(Q_j^k\cap \Omega_{k+\ell})}{\mu(Q_j^k)}
\le
C_\mathrm{D}\,a^{-\ell}.
\end{equation}
\end{lemma}
The Lebesgue differentiation theorem holds in a space of homogeneous 
type~$(\Omega,d,\mu)$ if and only if the measure~$\mu$ is 
\emph{Borel-semiregular}~\cite[Theorem~3.14]{AM15}, i.e.,
if every finite-measure set $E\subseteq\Omega$ differs from a Borel set~$B$ by a null 
set, meaning $\mu(E\triangle B)=0$. There is also a dyadic version of 
this theorem~\cite[Lemma~3.6(c)]{AS26}:
\begin{theorem}[Dyadic differentiation]\label{th:LDT}
Let~$(\Omega,d,\mu)$ be a space of homogeneous type with a Borel-semiregular 
measure~$\mu$, and let~$\mathscr{D}$ be a dyadic grid on~$\Omega$. For 
each~$x\in\Omega$, let~$Q_\ell(x)$ denote the unique cube in~$\mathscr{D}_\ell$
containing~$x$. If~$f\in L^1_\textnormal{loc}(\Omega,d,\mu)$, then
\[
\lim_{\ell\to+\infty}\fint_{Q_\ell(x)} f(y)\,d\mu(y)=f(x)
\quad\text{for $\mu$-almost every }x\in\Omega.
\]
\end{theorem}
The operators~$M$ and~$M^\mathscr{D}$ are linked via a finite system of 
\emph{adjacent} dyadic grids, constructed by 
Hyt\"onen and Kairema~\cite[Theorem~4.1]{HK12}. Serving as an analogue to the 
$3^n$~shifted grids in the Euclidean setting,
\[
\mathscr{D}^\mathbf{t}:=\left\{ 
2^{-k}\big([0,1)^n+\mathbf{m}+(-1)^k\mathbf{t}\big)\ :\ 
k\in\mathbb{Z},\, \mathbf{m}\in\mathbb{Z}^n\right\},
\quad\mathbf{t}\in\big\{0,\tfrac13,\tfrac23\big\}^n,
\]
this geometric construction provides two-sided pointwise bounds relating the standard 
maximal function to its dyadic counterparts~\cite[Proposition~7.9]{HK12}.
\begin{theorem}[Adjacent dyadic grids]\label{th:adjacent-grids}
Let~$(\Omega,d,\mu)$ be a space of homogeneous type. There exists a 
\emph{finite} collection of dyadic grids $\mathscr{D}^1,\ldots,\mathscr{D}^K$ 
on~$\Omega$---satisfying Theorem~\ref{th:HK} with the same parameters~$C_\mathrm{D}$, 
$R$, and~$\delta$---such that for all $f\in L^{1}_{\textnormal{loc}}(\Omega,d,\mu)$ 
and~$x\in\Omega$,
\[
Mf(x)\le C\sum_{t=1}^{K}M^{\mathscr{D}^{t}}f(x)
\quad\text{and}\quad
M^{\mathscr{D}^t}f(x)\le C\,Mf(x),
\]
where $K$ and $C\ge1$ depend only on the structural constants of~$(\Omega,d,\mu)$.
\end{theorem}

\section{Necessity of $p_->1$: An Atomic Classification}
\label{sec:p-=1}
We begin by addressing the lower endpoint~$p_-=1$ for the maximal operator~$M$ on 
variable Lebesgue spaces~$L^{p(\cdot)}(\Omega,d,\mu)$. In joint work with 
Karlovych~\cite[Theorem~1]{KS25}, we proved that if the measure~$\mu$ is 
Borel-semiregular and \emph{reverse doubling}, satisfying the bound
\[
\mu(B(x,r))\le c_\mu\,\mu(B(x,2r))
\]
for all balls~$B(x,r)$ with an absolute constant~$c_\mu\in(0,1)$, then 
the boundedness of~$M$ on~$L^{p(\cdot)}(\Omega,d,\mu)$ necessarily 
implies~$p_->1$. This initial result, however, left open 
whether the reverse doubling assumption could be relaxed.

Reverse doubling spaces cannot contain atoms~\cite[Lemma~2.8]{CUS18},
that is, points~$x\in\Omega$ with~$\mu(\{x\})>0$.
In the following theorem, we prove that the reverse doubling hypothesis can indeed be 
replaced by the weaker condition of~\emph{non-atomicity}. We achieve this by refining 
the counterexample from~\cite{KS25} to use dyadic rather than metric annuli.
\begin{theorem}[No atoms]\label{th:KS}
Let $(\Omega,d,\mu)$ be a \emph{non-atomic} space of homogeneous type with a 
Borel-semiregular measure~$\mu$. Given an exponent $p(\cdot)\in\mathcal{P}(\Omega,\mu)$,
if~$M$ is bounded on~$L^{p(\cdot)}(\Omega,d,\mu)$, then $p_->1$.
\end{theorem}
\begin{proof}
Assume to the contrary that~$p_-=1$.
Let~$\mathscr{D}=\bigcup_{\ell\in\mathbb{Z}}\mathscr{D}_\ell$ be any of the 
adjacent dyadic grids from Theorem~\ref{th:adjacent-grids}, so that the dyadic maximal 
operator~$M^\mathscr{D}$ is pointwise controlled by~$M$. It suffices to prove 
that~$M^\mathscr{D}$ is unbounded on~$L^{p(\cdot)}(\Omega,d,\mu)$.

Since $p_-=1$, for each~$k\in\mathbb{N}$ the set
\[
E_k:=\big\{x\in\Omega\,:\,p(x)<1+\tfrac1k\big\}
\]
has positive measure. Applying dyadic differentiation from
Theorem~\ref{th:LDT} to the function~$\mathbf{1}_{E_k}$, we can choose a
density point~$x_k\in E_k$ such that
\begin{equation}\label{lim:ratio}
\lim_{\ell\to+\infty}\frac{\mu(E_k\cap Q_\ell(x_k))}{\mu(Q_\ell(x_k))}=1,
\end{equation}
where $Q_\ell(x_k)$ is the unique cube in~$\mathscr{D}_\ell$ containing~$x_k$.
Let $C_\mathrm{D}$ be the dyadic doubling constant of~$\mathscr{D}$. Consider
only those generations $\ell\ge\ell_0$ for which the ratio in~\eqref{lim:ratio}
exceeds the value~$1-\frac{1}{2C_\mathrm{D}}$. Since the space is non-atomic, 
we have $\mu(Q_\ell(x_k))\to0$ as $\ell\to+\infty$.  
Thus the dyadic cone $\{Q_\ell(x_k)\}_{\ell\ge\ell_0}$ has infinitely many 
distinct cubes.

Enumerate the distinct cubes in the cone as
\[
Q_0\supsetneq Q_1\supsetneq Q_2\supsetneq\ldots.
\]
After omitting repeated copies of the same cube, $Q_{i+1}$ is 
a genuine child of $Q_i$ at the generation where the strict division occurs.
Consequently, $Q_i$ contains at least one other child~$Q'_{i+1}$. The dyadic doubling
condition $\mu(Q_i)\le C_\mathrm{D}\,\mu(Q'_{i+1})$ immediately forces a 
reverse doubling property:
\begin{equation}\label{est:RD}
\mu(Q_{i+1})\le\mu(Q_i)-\mu(Q'_{i+1})
\le \Big(1-\frac{1}{C_\mathrm{D}}\Big)\,\mu(Q_i).
\end{equation}

Define the dyadic annuli $A_i:=Q_i\setminus Q_{i+1}$ for all~$i\ge0$. By dyadic 
doubling, their measures are bounded from below:
\begin{equation}\label{est:annuli}
\mu(A_i)\ge\mu(Q'_{i+1})\ge\frac{\mu(Q_i)}{C_\mathrm{D}}
\ge\frac{\mu(Q_0)}{C_\mathrm{D}^{i+1}}.
\end{equation}
By our density choice, we have 
$\frac{\mu(E_k\cap Q_i)}{\mu(Q_i)}>1-\frac{1}{2C_\mathrm{D}}$. Combined with the 
reverse doubling estimate~\eqref{est:RD}, this yields the mass bound
\begin{equation}\label{est:port-annuli}
\frac{\mu(E_k\cap A_i)}{\mu(A_i)}
\ge\frac{\mu(E_k\cap Q_i)}{\mu(Q_i)}-\frac{\mu(Q_{i+1})}{\mu(Q_i)}
>\frac{1}{2C_\mathrm{D}}.
\end{equation}

Now, define the test functions
\[
f_k(x):=\sum_{i=0}^\infty\frac{\mathbf{1}_{E_k\cap A_i}(x)}
{C_\mathrm{D}^{\frac{i}{k}}\,\mu(A_i)},
\quad x\in\Omega.
\]
Note that $f_k$ is supported in~$E_k\cap Q_0$, and we have
$f_k(x)^{p(x)}\le1+f_k(x)^{1+\frac1k}$ for all $x\in E_k$.
Combined with the lower bound~\eqref{est:annuli}, this gives
\begin{align*}
\varrho_{p(\cdot)}(f_k)&=
\int_{E_k\cap Q_0} f_k(x)^{p(x)}\,d\mu(x)
\\
&\le\mu(Q_0)+\int_{E_k\cap Q_0}f_k(x)^{1+\frac1k}\,d\mu(x)
\\
&=\mu(Q_0)+\sum_{i=0}^\infty\frac{\mu(E_k\cap A_i)}
{\big[C_\mathrm{D}^\frac{i}{k}\,\mu(A_i)\big]^{1+\frac1k}}
\\
&\le\mu(Q_0)+\sum_{i=0}^\infty\frac{\mu(A_i)^{-\frac1k}}
{C_\mathrm{D}^{\frac{i}{k}+\frac{i}{k^2}}}
\\
&\le\mu(Q_0)+\Big[\frac{C_\mathrm{D}}{\mu(Q_0)}\Big]^{\frac1k}\,
\sum_{i=0}^\infty C_\mathrm{D}^{-\frac{i}{k^2}}
<\infty.
\end{align*}
Thus, $f_k\in L^{p(\cdot)}(\Omega,d,\mu)$ for all~$k$. Next, take~$x\in E_k\cap A_i$
for some~$i\ge0$. Since $x\in Q_i$, we use the mass bound~\eqref{est:port-annuli} to 
estimate the value of the dyadic maximal function at~$x$ by
\begin{align*}
M^\mathscr{D}f_k(x)&\ge\fint_{Q_i} f_k(y)\,d\mu(y)
\\
&=\frac{1}{\mu(Q_i)}\sum_{j=i}^\infty\frac{\mu(E_k\cap A_j)}
{C_\mathrm{D}^\frac{j}{k}\,\mu(A_j)}
\\
&\ge\frac{1}{\mu(Q_i)}\frac{1}{2C_\mathrm{D}}
\sum_{j=i}^\infty C_\mathrm{D}^{-\frac{j}{k}}
\\
&=\frac{1}{\mu(Q_i)}\frac{C_\mathrm{D}^{-\frac{i}{k}}}
{2C_\mathrm{D}\,\big(1-C_\mathrm{D}^{-\frac1k}\big)}.
\end{align*}
Using $\mu(A_i)\ge\frac{1}{C_\mathrm{D}}\,\mu(Q_i)$ from~\eqref{est:annuli}, we obtain
\[
M^\mathscr{D}f_k(x)\ge\frac{f_k(x)}
{2C_\mathrm{D}^2\,\big(1-C_\mathrm{D}^{-\frac1k}\big)}.
\]
Because the right-hand side is zero outside the support of~$f_k$, the inequality 
holds for all~$x\in\Omega$. Since $C_\mathrm{D}^{-\frac1k}\to1$ as~$k\to\infty$,
the norms~$\|M^\mathscr{D}f_k\|_{p(\cdot)}$ cannot be uniformly 
bounded by~$\|f_k\|_{p(\cdot)}$. Thus $M^\mathscr{D}$ is unbounded 
on~$L^{p(\cdot)}(\Omega,d,\mu)$, which completes the proof.
\end{proof}
In fact, non-atomicity is the sharp condition for the necessity of $p_->1$. As the next 
theorem shows, the presence of even a single atom destroys this necessity, allowing the 
maximal operator to remain bounded when $p_-=1$.
\begin{theorem}[At least one atom]\label{th:atom}
If a space of homogeneous type~$(\Omega,d,\mu)$ has an atom, there exists
an exponent $p(\cdot)\in\mathcal{P}(\Omega,\mu)$ with $p_-=1$ such that 
$M$~is \emph{bounded} on~$L^{p(\cdot)}(\Omega,d,\mu)$. 
\end{theorem}
\begin{proof}
Let $A:=\{a\}$ be an atom, that is, $0<\mu(A)<\infty$. The finiteness of~$\mu(A)$
follows automatically since balls have finite measure. Define 
$p(\cdot)\in\mathcal{P}(\Omega,\mu)$ by
\[
p(x)=\begin{cases}
1 & \text{if }x\in A,
\\
2 & \text{if }x\in\Omega\setminus A.
\end{cases}
\]
Clearly, $p_-=1$. To prove that~$M$ is bounded on~$L^{p(\cdot)}(\Omega,d,\mu)$, 
we use the following norm identities:
\[
\|f\mathbf{1}_A\|_{p(\cdot)}=\|f\mathbf{1}_A\|_1
\quad\text{and}\quad
\|f\mathbf{1}_{\Omega\setminus A}\|_{p(\cdot)}=
\|f\mathbf{1}_{\Omega\setminus A}\|_2,
\]
as well as
\[
\|f\mathbf{1}_A\|_2=|f(a)|\,\mu(A)^{\frac12}=\mu(A)^{-\frac12}\,\|f\mathbf{1}_A\|_1.
\]
Together, they imply the norm inequality 
\begin{align*}
\|f\|_2&\le\|f\mathbf{1}_A\|_2+\|f\mathbf{1}_{\Omega\setminus A}\|_2
\\
&=\mu(A)^{-\frac12}\,\|f\mathbf{1}_A\|_{p(\cdot)}
+\|f\mathbf{1}_{\Omega\setminus A}\|_{p(\cdot)}
\\
&\le (\mu(A)^{-\frac12}+1)\,\|f\|_{p(\cdot)}.
\end{align*}
Since~$M$ is bounded on~$L^2(\Omega,d,\mu)$ with some operator norm~$C_2$, we obtain
\begin{align*}
\|Mf\|_{p(\cdot)}&\le\|(Mf)\mathbf{1}_A\|_1
+\|(Mf)\mathbf{1}_{\Omega\setminus A}\|_2
\\
&=\mu(A)^{\frac12}\,\|(Mf)\mathbf{1}_A\|_2
+\|(Mf)\mathbf{1}_{\Omega\setminus A}\|_2
\\
&\le(\mu(A)^\frac12+1)\,\|Mf\|_2
\\
&\le C_2\,\frac{(\mu(A)^\frac12+1)^2}{\mu(A)^\frac12}\,\|f\|_{p(\cdot)}.
\end{align*}
This completes the proof.
\end{proof}
Theorems~\ref{th:KS} and~\ref{th:atom} provide a precise classification: 
a space~$(\Omega,d,\mu)$ with a Borel-semiregular measure~$\mu$ is 
non-atomic \emph{if and only if} $p_->1$ is necessary for the boundedness of~$M$ 
on~$L^{p(\cdot)}(\Omega,d,\mu)$ for every exponent 
$p(\cdot)\in\mathcal{P}(\Omega,\mu)$.

\section{A Jensen-Type Inequality and Dyadic $A_\infty$ Weights}
\label{sec:Jensen}
The rest of the paper focuses on the upper endpoint~$p_+=\infty$.
This section develops the machinery for the proof of our main 
Theorem~\ref{th:main-inf}, which addresses this limiting case. 
We begin with a Jensen-type inequality, extending the 
refinement~\cite[Lemma~5.2]{ADK26} of the original Euclidean 
inequality by Kopaliani~\cite[Lemma~2.4]{TK07} to spaces of homogeneous type.
\begin{lemma}\label{le:Jensen}
Let $q(\cdot)\in\mathcal{A}^\mathscr{D}_\textnormal{loc}(\Omega,d,\mu)$ with 
$1<q_-\le q_+<\infty$, and $\lambda:=\min\{\frac12,\frac{1}
{2[q]_{\mathcal{A}}^\mathscr{D}}\}$. 
If $f\in L^{q(\cdot)}(\Omega,d,\mu)$ satisfies $\|f\|_{q(\cdot)}\le\frac12$
and $|f|_Q\ge1$ for some cube $Q\in\mathscr{D}$, then 
\[
\int_Q (\lambda\,|f|_{Q})^{q(x)}\,d\mu(x)
\le
\int_Q |2f(x)|^{q(x)}\,d\mu(x).
\]
\end{lemma}
\begin{proof}
We use the dynamic parameter method due to Lerner~\cite[Lemma~3.1]{AL10}. 
Define 
\[
I(\alpha):=\fint_Q \alpha^{q'(y)-1}\,d\mu(y)
\quad\textnormal{for }\alpha>0.
\]
Since $1<(q')_-\le(q')_+<\infty$, the function $I(\alpha)$ is continuous and strictly 
increasing from $0$ to $\infty$. Therefore, choose $\alpha>0$ such that 
$I(\alpha)=|f|_Q$. The hypothesis $|f|_Q\ge 1$ implies that $\alpha\ge1$.
On the other hand, by H\"older's inequality and $\|f\|_{q(\cdot)}\le\frac12$, we have
$|f|_Q\le\frac{\|\mathbf{1}_Q\|_{q'(\cdot)}}{\mu(Q)}$. Since $(q')_+<\infty$,  
the unit sphere property gives
\[
I(\alpha)\le 
\|\mathbf{1}_Q\|_{q'(\cdot)} 
\fint_Q \bigg(\frac1{\|\mathbf{1}_Q\|_{q'(\cdot)}}\bigg)^{q'(y)}d\mu(y)
=\fint_Q \bigg(\frac{1}{\|\mathbf{1}_Q\|_{q'(\cdot)}}\bigg)^{q'(y)-1}d\mu(y).
\]
Comparing the integrands establishes an upper bound for $\alpha$. Altogether,
\begin{equation}\label{est:alpha}
1\le\alpha\le\frac{1}{\|\mathbf{1}_Q\|_{q'(\cdot)}}.
\end{equation}

Factorize the left-hand side of the desired inequality as
\begin{equation}\label{eq:factors}
\begin{aligned}
\int_Q (\lambda\,|f|_{Q})^{q(x)}\,d\mu(x) 
&=\int_Q \bigg(\lambda \fint_Q \alpha^{q'(y)-1}\,d\mu(y)\bigg)^{q(x)}d\mu(x)
\\
&=\int_Q \bigg(\lambda \fint_Q \alpha^{q'(y)-q'(x)}\,d\mu(y)\bigg)^{q(x)}
\alpha^{q'(x)}\,d\mu(x)
\\
&=\lambda \fint_{Q}
\bigg(\lambda \fint_Q \alpha^{q'(y)-q'(x)}\,d\mu(y)\bigg)^{q(x)-1}\,d\mu(x)
\ \int_{Q}\alpha^{q'(y)}\,d\mu(y).
\end{aligned}
\end{equation}
We show that the first factor is bounded by $1$. For any $x\in Q$, 
estimates~\eqref{est:alpha} imply
\[
\fint_Q \alpha^{q'(y)-q'(x)}\,d\mu(y)
\le\fint_Q \bigg[1+\bigg(\frac1{\|\mathbf{1}_Q\|_{q'(\cdot)}}\bigg)^{q'(y)-q'(x)} 
\bigg]\,d\mu(y)=1+\frac{\|\mathbf{1}_Q\|_{q'(\cdot)}^{q'(x)}}{\mu(Q)}.
\]
Substituting this bound and applying the inequality $(u+v)^r\le 2^r\,(u^r+v^r)$, in 
view of $q(\cdot)\in\mathcal{A}^\mathscr{D}_\textnormal{loc}(\Omega,d,\mu)$ and the 
definition of~$\lambda$, we obtain
\begin{align*}
\lambda \fint_{Q}
\bigg(\lambda \fint_Q \alpha^{q'(y)-q'(x)}\,&d\mu(y)\bigg)^{q(x)-1}\,d\mu(x)
\\
&\le\fint_{Q} \lambda^{q(x)}\,
\bigg[1+\frac{\|\mathbf{1}_Q\|_{q'(\cdot)}^{q'(x)}}{\mu(Q)}\bigg]^{q(x)-1}\,d\mu(x)
\\
&\le\frac12\fint_Q (2\lambda)^{q(x)}\,d\mu(x)
+\frac12\int_Q \bigg(\frac{2\lambda\,\|\mathbf{1}_Q\|_{q'(\cdot)}}
{\mu(Q)}\bigg)^{q(x)}\,d\mu(x)
\\
&\le\frac12+\frac12\int_Q \bigg(\frac{2\lambda[q]_\mathcal{A}^\mathscr{D}}
{\|\mathbf{1}_Q\|_{q(\cdot)}}\bigg)^{q(x)}\,d\mu(x)
\\
&\le\frac12+\frac12\int_Q 
\bigg(\frac{1}{\|\mathbf{1}_Q\|_{q(\cdot)}}\bigg)^{q(x)}\,d\mu(x)
\\
&=1.
\end{align*}

Representation~\eqref{eq:factors}, the identity $I(\alpha)=|f|_Q$, and Young's 
inequality $2\alpha\,|f(y)|\le\alpha^{q'(y)}+|2f(y)|^{q(y)}$ therefore yield
\begin{align*}
\int_Q (\lambda\,|f|_Q)^{q(x)}\,d\mu(x)
&\le \int_Q \alpha^{q'(y)}\,d\mu(y)
\\
&=2\alpha\int_{Q}|f(y)|\,d\mu(y)-\int_{Q}\alpha^{q'(y)}\,d\mu(y)
\\
&\le \int_Q \Big(\alpha^{q'(y)}+|2f(y)|^{q(y)}\Big)\,d\mu(y)
-\int_{Q}\alpha^{q'(y)}\,d\mu(y)
\\
&=\int_Q |2f(x)|^{q(x)}\,d\mu(x).
\end{align*}
This is the desired inequality.
\end{proof}
We use Lemma~\ref{le:Jensen} to show that specific variable-exponent weights 
belong to the following dyadic $A_\infty$ class. Recall that a weight is a 
nonnegative, locally integrable function on~$\Omega$. For any $\mu$-measurable set~$E$,
we adopt the notation $w(E):=\int_E w(x)\,d\mu(x)$.
\begin{definition}\label{def:A-inf}
A weight $w$ is a \emph{dyadic $A_\infty$ weight} on a cube $Q\in\mathscr{D}$,  
denoted $w\in A^\mathscr{D}_\infty(Q)$, if there exist $\alpha,\beta\in(0,1)$ such that 
for every $Q'\in\mathscr{D}$ with $Q'\subseteq Q$, 
and every $\mu$-measurable set $E\subseteq Q'$, 
\[
\mu(E)\le\alpha\,\mu(Q')
\quad\text{implies}\quad
w(E)\le\beta\,w(Q').
\]
The parameters $\alpha$ and $\beta$ are called the 
\emph{$A_\infty^\mathscr{D}$-constants} of $w$ on $Q$.
\end{definition}
Dyadic $A_\infty$ weights in spaces of homogeneous type admit an analogue of the 
classical Coifman--Fefferman power-decay characterization~\cite{CF74}. The non-trivial 
necessity direction in the following theorem is proved in~\cite[Theorem~4.3]{KST26}
via a Calder\'on--Zygmund argument and the Lebesgue differentiation theorem,
which requires the measure to be Borel-semiregular.
\begin{theorem}[Power decay for $A_\infty^\mathscr{D}$ weights]\label{th:equiv-Ainf}
Let $(\Omega,d,\mu)$ be a space of homogeneous type with a Borel-semiregular 
measure~$\mu$, and let~$\mathscr{D}$ be a dyadic grid on~$\Omega$.
For any~$Q\in\mathscr{D}$, one has $w\in A^\mathscr{D}_\infty(Q)$ 
\emph{if and only if} there exist $\eta\ge1$ and $\varepsilon\in(0,1)$ such that
for every $Q'\in\mathscr{D}$ with $Q'\subseteq Q$, 
and every $\mu$-measurable set $E\subseteq Q'$,
\[
\frac{w(E)}{w(Q')}\le \eta\bigg(\frac{\mu(E)}{\mu(Q')}\bigg)^\varepsilon.
\]
Moreover, if $w\in A^\mathscr{D}_\infty(Q)$, then~$\eta$ and~$\varepsilon$ 
depend only on the $A_\infty^\mathscr{D}$-constants of~$w$ on~$Q$ and the dyadic 
doubling constant~$C_\mathrm{D}$.
\end{theorem}
We now show that $(\lambda t)^{q(\cdot)}\in A_\infty^\mathscr{D}(Q)$ under suitable 
assumptions on the exponent~$q(\cdot)$ and the scalars~$\lambda$ and~$t$, 
and deduce the corresponding power-decay inequality.
\begin{lemma}[Uniform power decay for $(\lambda t)^{q(\cdot)}$]\label{le:power-decay}
Let $(\Omega,d,\mu)$ be a space of homogeneous type with a Borel-semiregular 
measure~$\mu$, and let~$\mathscr{D}$ be a dyadic grid on~$\Omega$. 
For~$q(\cdot)\in\mathcal{A}^\mathscr{D}_\textnormal{loc}(\Omega,d,\mu)$   
with $1<q_-\le q_+<\infty$, set $\lambda:=\min\{\frac12,\frac{1}
{2[q]_{\mathcal{A}}^\mathscr{D}}\}$. Suppose a cube~$Q\in\mathscr{D}$ and
a scalar~$t$ satisfy
\[
1\le t\le\frac{[q]_\mathcal{A}^\mathscr{D}}{\|\mathbf{1}_Q\|_{q(\cdot)}}.
\]
Then there exist constants $\eta\ge1$ and $\varepsilon\in(0,1)$, depending only on 
$q_+$, $[q]_\mathcal{A}^\mathscr{D}$, and the dyadic doubling constant~$C_\mathrm{D}$, 
such that for any $\mu$-measurable set $E\subseteq Q$,
\[
\int_E (\lambda t)^{q(x)}\,d\mu(x)\le 
\eta \bigg(\frac{\mu(E)}{\mu(Q)}\bigg)^\varepsilon
\int_Q (\lambda t)^{q(x)}\,d\mu(x).
\]
\end{lemma}
\begin{proof}
We first establish that $w:=(\lambda t)^{q(\cdot)}$ is a dyadic $A_\infty$ weight 
on $Q$. Let $E\subseteq Q'\subseteq Q$ with $Q'\in\mathscr{D}$ and 
$\mu(E)\ge\mu(Q')/2$. Define $f:=\lambda t\,\mathbf{1}_E$. By assumption, 
\begin{align*}
\|f\|_{q(\cdot)}&=\lambda t\,\|\mathbf{1}_E\|_{q(\cdot)}
\le \lambda [q]_\mathcal{A}^\mathscr{D}\le\frac12,
\\
|f|_{Q'}&=\frac{\mu(E)}{\mu(Q')}\,\lambda t\ge\frac{\lambda t}{2}.
\end{align*}
We distinguish two cases, keeping in mind that $[q]_\mathcal{A}^\mathscr{D}\ge\frac12$ 
implies $1/\lambda\le 4[q]_\mathcal{A}^\mathscr{D}$.

First, suppose $t\ge2/\lambda$, which ensures $|f|_{Q'}\ge 1$. Applying
Lemma~\ref{le:Jensen} yields
\begin{align*}
w(E)
&\ge 2^{-q_+}\int_{Q'} |2f(x)|^{q(x)}\,d\mu(x)
\ge 2^{-q_+}\int_{Q'} (\lambda\,|f|_{Q'})^{q(x)}\,d\mu(x)
\\
&\ge \Big(\frac{\lambda}{4}\Big)^{q_+} \int_{Q'} (\lambda t)^{q(x)}\,d\mu(x)
\ge (16[q]_\mathcal{A}^\mathscr{D})^{-q_+}\, w(Q').
\end{align*}
Second, suppose $1\le t<2/\lambda$. Since $\lambda< 1$, we have 
$(\lambda t)^{q(x)}\ge\lambda^{q(x)}\ge\lambda^{q_+}$ pointwise. Using this 
and the assumption $\lambda t/2 < 1$, we estimate
\begin{align*}
w(E)&\ge\lambda^{q_+}\,\mu(E)\ge\frac{\lambda^{q_+}}{2}\,\mu(Q')
\ge\frac{\lambda^{q_+}}{2} \int_{Q'} \Big(\frac{\lambda t}{2}\Big)^{q(x)}\,d\mu(x)
\\
&\ge \Big(\frac{\lambda}{4}\Big)^{q_+} \int_{Q'} (\lambda t)^{q(x)}\,d\mu(x)
\ge (16[q]_\mathcal{A}^\mathscr{D})^{-q_+}\, w(Q').
\end{align*}
By Definition~\ref{def:A-inf}, we proved that $w\in A_\infty^\mathscr{D}(Q)$
with the constants $\alpha=\frac12$ and 
$\beta=1-(16[q]_\mathcal{A}^\mathscr{D})^{-q_+}$.
Thus, Theorem~\ref{th:equiv-Ainf} guarantees the desired power-decay inequality
with constants $\eta$ and $\varepsilon$ depending only on~$\beta$ and~$C_\mathrm{D}$.
\end{proof}
\begin{remark}
Lemma~\ref{le:power-decay} is designed to replace~\cite[Lemma~5.3]{ADK26} in our 
generalization of the Euclidean result~\cite[Proposition~5.4]{ADK26}. The original 
proof contained an unjustified step: the transition from~(5.21) to~(5.22) implicitly 
assumed that $(\lambda\,\alpha_{j,k}(g))^{p'(\cdot)}$ belongs to~$A_\infty(Q_j^k)$, 
whereas~\cite[Lemma~5.3]{ADK26} only established this for the unscaled weight 
$(\alpha_{j,k}(g))^{p'(\cdot)}$. Our modified lemma removes this gap from the 
generalized Theorem~\ref{th:aux-US}.
\end{remark}

\section{Boundedness of Dyadic Maximal Operators for $p_+<\infty$}
\label{sec:p+<inf}
We now establish a restricted version of Theorem~\ref{th:main-inf} for 
the dyadic maximal operator~$M^\mathscr{D}$ and bounded exponents. We treat 
\emph{unbounded} and \emph{bounded} spaces of homogeneous type separately, as the
latter setting requires weaker assumptions. 
Our first theorem addresses unbounded spaces, following the core argument 
from~\cite[Proposition~5.4]{ADK26} but systematically simplifying 
its technical execution and presentation.
\begin{theorem}[For unbounded spaces and $p_+<\infty$]\label{th:aux-US} 
Let $(\Omega,d,\mu)$ be 
an \emph{unbounded} space of homogeneous type with a Borel-semiregular measure~$\mu$, 
and let~$\mathscr{D}$ be a dyadic grid on~$\Omega$. Suppose that 
\[
p(\cdot)\in\mathcal{A}^\mathscr{D}_\textnormal{loc}(\Omega,d,\mu)\cap\mathcal{N}(\Omega,\mu)
\quad\text{and}\quad
1<p_-\le p_+<\infty.
\]
Then for all $f\in L^{p(\cdot)}(\Omega,d,\mu)$, 
\begin{equation}\label{ineq:M-D}
\|M^\mathscr{D}f\|_{p(\cdot)}\le c\,\|f\|_{p(\cdot)},
\end{equation}
where $c>0$ depends only on $(p')_+$, $[p]_\mathcal{A}^\mathscr{D}$, 
$[p]_\mathcal{N}$, and the dyadic doubling constant~$C_\mathrm{D}$.
\end{theorem}
\begin{proof}
By the Fatou property applied to the truncations $f_k:=\min\{|f|,k\}\,\mathbf{1}_{B_k}$ 
over an expanding sequence of balls~$B_k$, it suffices to prove~\eqref{ineq:M-D} 
assuming $f$ is bounded and supported in a ball. By the norm conjugate formula
from Theorem~\ref{th:toolbox}(e), this reduces to showing
\begin{equation}\label{ineq:M-D-int}
\int_\Omega M^\mathscr{D}f(x)\,g(x)\,d\mu(x)\le c\,\|f\|_{p(\cdot)}
\end{equation}
for all nonnegative, bounded $g\in L^{p'(\cdot)}(\Omega,d,\mu)$ supported in a 
ball with
\[
\|g\|_{p'(\cdot)}\le\frac12.
\]

Our assumptions ensure 
$f\in L^1(\Omega,d,\mu)$, hence Lemma~\ref{le:CZ} is applied. Fix
$a>2\,C_\mathrm{D}$. For each $k\in\mathbb{Z}$, the level set 
\[
\Omega_k:=\big\{x\in\Omega\,:\,M^\mathscr{D}f(x)>a^k\big\}
=\bigcup_{j}Q_j^k
\]
admits a disjoint decomposition into maximal cubes $Q_j^k\in\mathscr{D}$
such that $|f|_{Q_j^k}>a^k$. Defining $E_j^k:=Q_j^k\setminus\Omega_{k+1}$, the sets 
$\{E_j^k\}_{k,j}$ form a disjoint partition of $\{x\in\Omega\,:\;M^\mathscr{D}f(x)>0\}$.
Introduce the local averages
\[
\alpha_j^k(g):=\frac{1}{\mu(Q_j^k)}\int_{E_j^k} g(y)\,d\mu(y)
\]
and define the auxiliary operator $T$ by
\[
Tg(x):=\sum_{k,j} \alpha_j^k(g)\,\mathbf{1}_{Q_j^k}(x),
\quad x\in\Omega.
\]
Using the pointwise bound $M^\mathscr{D}f\le a^{k+1}<a\,|f|_{Q_j^k}$ on $E_j^k$ 
and H\"older's inequality yields
\begin{align*}
\int_\Omega M^\mathscr{D}f(x)\,g(x)\,d\mu(x)&=
\sum_{k,j} \int_{E_j^k}M^\mathscr{D}f(x)\,g(x)\,d\mu(x)
\\
&\le a\sum_{k,j}
\bigg(
\frac1{\mu(Q_j^k)}\int_{Q_j^k} |f(y)|\,d\mu(y)
\bigg)
\int_{E_j^k}g(x)\,d\mu(x)
\\
&=a\sum_{k,j} \alpha_j^k(g)
\int_\Omega |f(y)|\,\mathbf{1}_{Q_j^k}(y)\,d\mu(y)
\\
&=a\int_\Omega |f(y)|\,Tg(y)\,d\mu(y)
\\
&\le
2a\,\|f\|_{p(\cdot)}\|Tg\|_{p'(\cdot)}.
\end{align*}

Consequently, establishing~\eqref{ineq:M-D-int} reduces to bounding
\begin{equation}\label{ineq:T-est}
\|Tg\|_{p'(\cdot)}\le c,
\end{equation}
where~$c>0$ depends only on $(p')_+$, $[p]_\mathcal{A}^\mathscr{D}$, 
$[p]_\mathcal{N}$, and~$C_\mathrm{D}$. To evaluate this, we slice the operator~$T$
along the annuli $D_{k+\ell}:=\Omega_{k+\ell}\setminus\Omega_{k+\ell+1}$ for 
$\ell\ge0$. Relative to the fixed function~$g$, we define the index sets
\[
\mathscr{I}^+:=\{(k,j):\alpha_j^k(g)>1\}
\quad\textnormal{and}\quad
\mathscr{I}^-:=\{(k,j):\alpha_j^k(g)\le1\}.
\]
Since $\Omega_k=\bigcup_{\ell=0}^\infty D_{k+\ell}$, we can decompose
$Tg=\sum_{\ell=0}^\infty(T_\ell^+g+T_\ell^-g)$, where the restricted linear operators
are defined for any locally integrable $h\ge0$ by
\[
T_\ell^+h:=\sum_{(k,j)\in\mathscr{I}^+} \alpha_j^k(h)\,
\mathbf{1}_{Q_j^k\cap D_{k+\ell}} 
\quad\text{and}\quad 
T_\ell^-h:=\sum_{(k,j)\in\mathscr{I}^-} \alpha_j^k(h)\,
\mathbf{1}_{Q_j^k\cap D_{k+\ell}}.
\]

\paragraph*{Case 1: $(k,j)\in\mathscr{I}^+$.}
By H\"older's inequality, 
$\|g\|_{p'(\cdot)}\le\frac12$, and 
$p(\cdot)\in\mathcal{A}^\mathscr{D}_\textnormal{loc}(\Omega,d,\mu)$,
\[
\alpha_j^k(g)\le
\frac{2}{\mu(Q_j^k)}\,\|g\mathbf{1}_{E_j^k}\|_{p'(\cdot)} 
\|\mathbf{1}_{Q_j^k}\|_{p(\cdot)}
\le\frac{[p]_\mathcal{A}^\mathscr{D}}{\|\mathbf{1}_{Q_j^k}\|_{p'(\cdot)}}.
\]
Setting $\lambda:=\min\{\frac12,\frac1{2[p]_\mathcal{A}^\mathscr{D}}\}$, we apply 
Lemma~\ref{le:power-decay} to the weight $(\lambda\,\alpha_j^k(g))^{p'(\cdot)}$,
yielding power-decay constants~$\eta\ge1$ and $\varepsilon\in(0,1)$ that depend only on 
$(p')_+$, $[p]_\mathcal{A}^\mathscr{D}$, and~$C_\mathrm{D}$
(independent of $k$ and $j$). Crucially, estimate~\eqref{ineq:decay} provides the 
geometric bound
\begin{equation}\label{est:meas-ratio}
\frac{\mu(Q_j^k\cap D_{k+\ell})}{\mu(Q_j^k)}\le C_\mathrm{D}\,a^{-\ell}.
\end{equation}
Proceeding with Lemma~\ref{le:Jensen} (applied to $f:=g\mathbf{1}_{E_j^k}$) allows 
us to bound the semimodular:
\begin{align*}
\varrho_{p'(\cdot)}(\lambda\,T_\ell^+g)
&=\sum_{(k,j)\in\mathscr{I}^+} 
\int_{Q_j^k\cap D_{k+\ell}} (\lambda\,\alpha_j^k(g))^{p'(x)}\,d\mu(x)
\\
&\le \eta\,(C_\mathrm{D}\,a^{-\ell})^\varepsilon
\sum_{(k,j)\in\mathscr{I}^+} 
\int_{Q_j^k} (\lambda\,\alpha_j^k(g))^{p'(x)}\,d\mu(x)
\\
&\le C_\mathrm{D}\,\eta\,a^{-\varepsilon\ell} 
\sum_{(k,j)\in\mathscr{I}^+} 
\int_{E_j^k} (2g(x))^{p'(x)}\,d\mu(x)
\\
&\le C_\mathrm{D}\,\eta\,a^{-\varepsilon\ell},
\end{align*}
where the final inequality holds because $\varrho_{p'(\cdot)}(2g)\le1$. Passing to
the norm via Theorem~\ref{th:toolbox}(c) and using 
$1/\lambda\le 4[p]_\mathcal{A}^\mathscr{D}$ yields the exponential decay
\begin{equation}\label{est:T+}
\|T_\ell^+g\|_{p'(\cdot)}
\le 4[p]_\mathcal{A}^\mathscr{D}\,C_\mathrm{D}\, 
\eta\,a^{-\frac{\varepsilon\ell}{(p')_+}}.
\end{equation}

\paragraph*{Case 2: $(k,j)\in\mathscr{I}^-$.} Denote $g_1:=g\mathbf{1}_{\{g\le1\}}$
and $g_2:=g\mathbf{1}_{\{g>1\}}$, and split 
\[
T_\ell^-g=T_\ell^-g_1+T_\ell^-g_2.
\]
For any constant exponent $q$ satisfying $1<q\le(p')_+$, 
Jensen's inequality and the geometric bound~\eqref{est:meas-ratio} yield
\[
\|T_\ell^-g_1\|_q
\le \bigg[ \sum_{(k,j)\in\mathscr{I}^-} \frac{\mu(Q_j^k\cap D_{k+\ell})}{\mu(Q_j^k)}
\int_{E_j^k} (g_1(y))^q\,d\mu(y)\bigg]^{\frac1q}
\le C_\mathrm{D}\,a^{-\frac{\ell}{(p')_+}}\,\|g_1\|_q.
\]
Let $p'_\infty$ be the Nekvinda parameter for 
$p'(\cdot)\in\mathcal{N}(\Omega,\mu)$. Recall that $(p')_-\le p'_\infty\le(p')_+$ and 
$[p]_\mathcal{N}=[p']_\mathcal{N}$. Applying Lemma~\ref{le:N-emb} to the
exponent~$p'(\cdot)$, we 
further obtain
\begin{align}\label{est:T-g1}
\|T_\ell^-g_1\|_{p'(\cdot)}
&\le 4[p]_\mathcal{N}\,
\max\big\{\|T_\ell^-g_1\|_{p'_\infty},\|T_\ell^-g_1\|_{(p')_+}\big\}
\nonumber
\\
&\le 4[p]_\mathcal{N}\, C_\mathrm{D}\,a^{-\frac{\ell}{(p')_+}}\,
\max\{\|g_1\|_{p'_\infty},\|g_1\|_{(p')_+}\}
\nonumber
\\
&\le 12[p]_\mathcal{N}^2\,C_\mathrm{D}\,a^{-\frac{\ell}{(p')_+}}.
\end{align}
For $T_\ell^-g_2$, the condition $\alpha_j^k(g_2)\le 1$ allows us to bound the 
semimodular via the same geometric decay.
Since the supports $\{Q_j^k\cap D_{k+\ell}\}_{k,j}$ are pairwise disjoint for a 
fixed~$\ell$,
\begin{align*}
\varrho_{p'(\cdot)}(T_\ell^-g_2)
&\le\sum_{(k,j)\in\mathscr{I}^-} \int_{Q_j^k\cap D_{k+\ell}}
\alpha_j^k(g_2)\,d\mu(x)
\\
&=\sum_{(k,j)\in\mathscr{I}^-} \frac{\mu(Q_j^k\cap D_{k+\ell})}{\mu(Q_j^k)}
\int_{E_j^k} g_2(x)\,d\mu(x)
\\
&\le C_\mathrm{D}\,a^{-\ell}
\int_\Omega g_2(x)^{p'(x)}\,d\mu(x)
\\
&\le C_\mathrm{D}\,a^{-\ell}.
\end{align*}
Theorem~\ref{th:toolbox}(c) then implies
\begin{equation}\label{est:T-g2}
\|T_\ell^-g_2\|_{p'(\cdot)}\le C_\mathrm{D}\,
a^{-\frac{\ell}{(p')_+}}.
\end{equation}

Finally, it follows from estimates~\eqref{est:T+}, \eqref{est:T-g1}, 
and~\eqref{est:T-g2} that 
\[
\|T_\ell^+g\|_{p'(\cdot)}+\|T_\ell^-g_1\|_{p'(\cdot)}
+\|T_\ell^-g_2\|_{p'(\cdot)}
\le
(4[p]_\mathcal{A}^\mathscr{D}\,\eta+12[p]_\mathcal{N}^2+1)\,
C_\mathrm{D}\, a^{-\frac{\varepsilon\ell}{(p')_+}}.
\]
By the Riesz--Fischer property from Theorem~\ref{th:toolbox}(h), summing 
over $\ell\ge0$ establishes
\begin{equation}\label{est:T}
\|Tg\|_{p'(\cdot)} \le
\frac{(4[p]_\mathcal{A}^\mathscr{D}\,\eta+12[p]_\mathcal{N}^2+1)\,C_\mathrm{D}}
{1-a^{-\frac{\varepsilon}{(p')_+}}},
\end{equation}
which is the desired bound~\eqref{ineq:T-est} with the constant in the right-hand 
side depending only on $(p')_+$, $[p]_\mathcal{A}^\mathscr{D}$, $[p]_\mathcal{N}$, 
and~$C_\mathrm{D}$. This completes the proof.
\end{proof}
Unlike the Euclidean proof in~\cite[Proposition~5.4]{ADK26}, 
our proof of Theorem~\ref{th:aux-US} explicitly tracks the dependence of the 
constant~$c$ on~$(p')_+$, $[p]_\mathcal{A}^\mathscr{D}$, and $[p]_\mathcal{N}$.
This precision is required to address a subtle gap in the original proof of the main 
result~\cite[Theorem~1.1]{ADK26}, which implicitly assumed that~$c$ depends only on the 
\emph{upper bounds} of these characteristics. To rigorously justify this, 
we establish the following corollary. It is this corollary, rather 
than Theorem~\ref{th:aux-US} itself, that we apply in the proof of our main 
Theorem~\ref{th:main-US}.
\begin{corollary}\label{cor:aux-US}
Let a sequence of exponents~$p_n(\cdot)$ satisfy the hypotheses of 
Theorem~\ref{th:aux-US} with uniform bounds
$(p'_n)_+\le c_+$, $[p_n]_\mathcal{A}^\mathscr{D}\le c_\mathcal{A}^\mathscr{D}$, and 
$[p_n]_\mathcal{N}\le c_\mathcal{N}$. Then for all $n\in\mathbb{N}$ and 
$f\in L^{p_n(\cdot)}(\Omega,d,\mu)$,
\[
\|M^\mathscr{D}f\|_{p_n(\cdot)}\le c\,\|f\|_{p_n(\cdot)}
\]
holds with a constant~$c>0$ depending 
only on $c_+$, $c_\mathcal{A}^\mathscr{D}$, $c_\mathcal{N}$, and~$C_\mathrm{D}$.
\end{corollary}
\begin{proof}
By the proof of Lemma~\ref{le:power-decay}, the weight  
$(\lambda\,\alpha_j^k(g))^{p'_n(\cdot)}$ belongs to $A_\infty^\mathscr{D}(Q_j^k)$ with 
constants $\alpha=\frac12$ and 
$\beta_n=1-(16[p_n]_\mathcal{A}^\mathscr{D})^{-(p'_n)_+}$. Because 
$\beta_n\le1-(16c_\mathcal{A}^\mathscr{D})^{-c_+}=:\beta$, 
Definition~\ref{def:A-inf} ensures this weight also satisfies the 
$A_\infty^\mathscr{D}$ condition for the uniform upper bound~$\beta$.
Applying Theorem~\ref{th:equiv-Ainf} with~$\alpha$ and~$\beta$ yields uniform 
power-decay constants $\eta\ge 1$ and $\varepsilon\in(0,1)$ that depend exclusively 
on~$c_+$, $c_\mathcal{A}^\mathscr{D}$, and~$C_\mathrm{D}$.

Substituting these uniform parameters into the proof of Theorem~\ref{th:aux-US}, and
combining the norm conjugate formula with the culmination bound~\eqref{est:T}, gives
\[
\|M^\mathscr{D}f\|_{p_n(\cdot)}
\le\frac{8a\,
(4[p_n]_\mathcal{A}^\mathscr{D}\,\eta+12[p_n]_\mathcal{N}^2+1)\,C_\mathrm{D}}
{1-a^{-\frac{\varepsilon}{(p'_n)_+}}}\,\|f\|_{p_n(\cdot)},
\]
where~$a>1$ depends only on~$C_\mathrm{D}$. Since the fractional multiplier is strictly 
non-decreasing with respect to $(p'_n)_+$, $[p_n]_\mathcal{A}^\mathscr{D}$, and 
$[p_n]_\mathcal{N}$, we can bound it uniformly by replacing these characteristics with 
their respective upper bounds $c_+$, $c_\mathcal{A}^\mathscr{D}$, and $c_\mathcal{N}$. 
This produces the desired uniform constant~$c$, completing the proof.
\end{proof}
For bounded spaces of homogeneous type, the Nekvinda condition is no 
longer required, allowing us to establish the following result.
\begin{proposition}[For bounded spaces and $p_+<\infty$]\label{prop:aux-BS}
Let $(\Omega,d,\mu)$ be 
a \emph{bounded} space of homogeneous type with a Borel-semiregular measure~$\mu$, 
and let~$\mathscr{D}$ be a dyadic grid on~$\Omega$. Suppose that 
\[
p(\cdot)\in\mathcal{A}^\mathscr{D}_\textnormal{loc}(\Omega,d,\mu)
\quad\text{and}\quad
1<p_-\le p_+<\infty.
\]
Then for all $f\in L^{p(\cdot)}(\Omega,d,\mu)$, 
\[
\|M^\mathscr{D}f\|_{p(\cdot)}\le c\,\|f\|_{p(\cdot)},
\]
where $c>0$ depends only on $(p')_+$, $[p]_\mathcal{A}^\mathscr{D}$, 
the dyadic doubling constant~$C_\mathrm{D}$, and the finite measure~$\mu(\Omega)$.
\end{proposition}
\begin{proof}
The argument adapts the proof of Theorem~\ref{th:aux-US}.
Assuming without loss of generality that~$f$ is bounded and 
satisfies~$\|f\|_{p(\cdot)}\le\frac12$, 
the norm conjugate formula reduces the problem to showing
\[
\int_\Omega M^\mathscr{D}f(x)\,g(x)\,d\mu(x)\le c
\]
for any nonnegative, bounded $g \in L^{p'(\cdot)}(\Omega,d,\mu)$ 
with~$\|g\|_{p'(\cdot)}\le\frac12$. 

Since $\Omega$ is bounded, it constitutes a 
single maximal cube in~$\mathscr{D}$. H\"older's inequality and 
$p(\cdot)\in\mathcal{A}_\textnormal{loc}^\mathscr{D}(\Omega,d,\mu)$
therefore yield
\[
|f|_\Omega\le\frac2{\mu(\Omega)}\,\|f\|_{p(\cdot)}\|\mathbf{1}_\Omega\|_{p'(\cdot)}
\le\frac{[p]_\mathcal{A}^\mathscr{D}}{\|\mathbf{1}_\Omega\|_{p(\cdot)}}=:c_0.
\]
Fixing $a>2\,C_\mathrm{D}$, we choose $k_0\in\mathbb{Z}$ such that 
$a^{k_0-1}\le c_0<a^{k_0}$ and split the integration over 
$\Omega\setminus\Omega_{k_0}$ and $\Omega_{k_0}$. 
On the complement $\Omega\setminus\Omega_{k_0}$, the pointwise bound 
$M^\mathscr{D}f\le a^{k_0}\le ac_0$ combined with H\"older's inequality implies
\[
\int_{\Omega\setminus\Omega_{k_0}} M^\mathscr{D}f(x)\,g(x)\,d\mu(x)
\le 2ac_0\,\|\mathbf{1}_\Omega\|_{p(\cdot)}\|g\|_{p'(\cdot)}
\le a[p]_\mathcal{A}^\mathscr{D}.
\]

On~$\Omega_{k_0}=\bigcup_{k\ge k_0}\Omega_k$, the condition
$a^k\ge a^{k_0}>c_0\ge|f|_\Omega$ ensures the level-set thresholds strictly exceed 
the global average. Applying the Calder\'on--Zygmund decomposition from 
Lemma~\ref{le:CZ} exclusively to generations $k\ge k_0$ produces the truncated 
auxiliary operator
\[
Tg(x) := \sum_{k\ge k_0} \sum_j \alpha_j^k(g)\,\mathbf{1}_{Q_j^k}(x),
\quad x\in\Omega,
\]
which, mirroring the unbounded case, yields the estimate
\[
\int_{\Omega_{k_0}} M^\mathscr{D}f(x)\,g(x)\,d\mu(x)
\le a\,\|Tg\|_{p'(\cdot)}.
\]

To bound~$\|Tg\|_{p'(\cdot)}$, we decompose 
$Tg=\sum_{\ell=0}^\infty (T_\ell^+g + T_\ell^-g)$. The large-average term~$T_\ell^+g$ 
satisfies the exponential decay~\eqref{est:T+}, exactly as in \emph{Case 1} of 
Theorem~\ref{th:aux-US}. For the small-average term~$T_\ell^-g$, the condition 
$\alpha_j^k(g) \le 1$ allows direct evaluation of the semimodular without splitting 
the test function~$g$:
\begin{align*}
\varrho_{p'(\cdot)}(T_\ell^-g)
&\le \sum_{k\ge k_0}\sum_{\substack{j:\\\alpha_j^k(g)\le1}}
\int_{Q_j^k\cap D_{k+\ell}} \alpha_j^k(g)\,d\mu(x)
\\
&\le C_\mathrm{D}\,a^{-\ell} \sum_{k\ge k_0}\sum_j 
\int_{E_j^k} g(x)\,d\mu(x)
\\
&\le C_\mathrm{D}\,a^{-\ell} \int_\Omega g(x)\,d\mu(x)
\\
&\le C_\mathrm{D}\,\|\mathbf{1}_\Omega\|_{p(\cdot)}\,a^{-\ell}.
\end{align*}
Since $1\le p(\cdot)\le\infty$, the squeezed embedding from 
Theorem~\ref{th:toolbox}(f) ensures
\[
\|\mathbf{1}_\Omega\|_{p(\cdot)} 
\le 2\,\max\{\|\mathbf{1}_\Omega\|_1, \|\mathbf{1}_\Omega\|_\infty\} 
= 2\,\max\{\mu(\Omega), 1\}.
\]
Applying the modular-to-norm estimate translates this into the 
exponential decay
\[
\|T_\ell^-g\|_{p'(\cdot)}\le2C_\mathrm{D}\,\max\{\mu(\Omega),1\}
\,a^{-\frac{\ell}{(p')_+}}.
\]
Summing the geometric series over $\ell\ge 0$ via the Riesz--Fischer property 
establishes $\|Tg\|_{p'(\cdot)}\le c$,
completing the proof with a final constant dependent only on~$(p')_+$, 
$[p]_\mathcal{A}^\mathscr{D}$, $C_\mathrm{D}$, and $\mu(\Omega)$.
\end{proof}
\begin{remark}
As in Corollary~\ref{cor:aux-US} for the unbounded case, one can replace 
the characteristics with their upper bounds $(p')_+\le c_+$ and 
$[p]_\mathcal{A}^\mathscr{D}\le c_\mathcal{A}^\mathscr{D}$,
and conclude that the norm inequality in Proposition~\ref{prop:aux-BS} holds
with a constant~$c$ depending only on~$c_+$, $c_\mathcal{A}^\mathscr{D}$, 
$C_\mathrm{D}$, and~$\mu(\Omega)$.
\end{remark}

\section{The Main Theorems for Unbounded Exponents}
\label{sec:main-ADK}
This section deduces Theorem~\ref{th:main-inf} from the restricted
Theorem~\ref{th:aux-US} and Proposition~\ref{prop:aux-BS} via suitable approximations 
of variable exponents that preserve the dyadic Muckenhoupt and Nekvinda conditions. 
The following theorem provides an analogue of~\cite[Lemma~4.4]{ADK26} for general
measure spaces.
\begin{theorem}[Approximations of variable exponents]\label{th:approximation}
Let $p(\cdot)\in\mathcal{P}(\Omega,\mu)$ and, for $n\in\mathbb{N}$, define
$p_n(\cdot)\in\mathcal{P}(\Omega,\mu)$ by
\[
\frac{1}{p_n(x)}:=\frac{1}{n+1}+\frac{n-1}{n+1}\,\frac{1}{p(x)},
\quad x\in\Omega.
\]
Then:
\begin{enumerate}
\item[\textnormal{(a)}] $p_n(\cdot)\to p(\cdot)$ $\mu$-almost everywhere;

\item[\textnormal{(b)}] $1<(p_n)_-\le(p_n)_+<\infty$ and $(p'_n)_+\le\max\{(p')_+,2\}$;

\item[\textnormal{(c)}] if 
$p(\cdot)\in\mathcal{A}_\textnormal{loc}(\Omega,d,\mu)$, then 
$p_n(\cdot)\in\mathcal{A}_\textnormal{loc}(\Omega,d,\mu)$ with 
$[p_n]_\mathcal{A}\le 8[p]_\mathcal{A}$;

\item[\textnormal{(d)}] if $p(\cdot)\in\mathcal{N}(\Omega,\mu)$, then 
$p_n(\cdot)\in\mathcal{N}(\Omega,\mu)$ with $[p_n]_\mathcal{N}\le2[p]_\mathcal{N}$;

\item[\textnormal{(e)}] for all $f\in L^2(\Omega,\mu)$,
\[
\lim_{n\to\infty}\|f\|_{\tilde{\varrho}_{p_n(\cdot)}}=
\|f\|_{\tilde{\varrho}_{p(\cdot)}}.
\]
\end{enumerate}
\end{theorem}
\begin{proof}
Convergence~(a) is immediate. The simple inequality 
$1+\frac1n\le p_n(x)<n+1$, $x\in\Omega$, implies that each $p_n(\cdot)$ is bounded
away from $1$ and $\infty$. Rewriting the definition as a convex combination,
\begin{equation}\label{eq:whm}
\frac{1}{p_n(x)}=\frac{2}{n+1}\,\frac12 +\frac{n-1}{n+1}\,\frac{1}{p(x)},
\end{equation}
shows that $p_n(x)$ lies between $p(x)$ and $2$. Consequently, $p'_n(x)$ lies between 
$p'(x)$ and $2$, hence $(p'_n)_+\le\max\{(p')_+,2\}$ and~(b) is established.

We now check~(c). Denote $\theta:=\frac{2}{n+1}$. By 
the interpolation estimate and rescaling in Theorem~\ref{th:toolbox}(i)--(j), 
identity~\eqref{eq:whm} implies that for any ball $B\subseteq\Omega$,
\[
\|\mathbf{1}_B\|_{p_n(\cdot)}\le2\,\|\mathbf{1}_B\|_{2}^\theta\,
\|\mathbf{1}_B\|_{p(\cdot)}^{1-\theta}
=2\,\mu(B)^{\theta/2}\,\|\mathbf{1}_B\|_{p(\cdot)}^{1-\theta}.
\]
Since the relation~\eqref{eq:whm} also holds for the conjugate exponents, we 
symmetrically obtain 
\[
\|\mathbf{1}_B\|_{p'_n(\cdot)}\le
2\,\mu(B)^{\theta/2}\,\|\mathbf{1}_B\|_{p'(\cdot)}^{1-\theta}.
\]
Together, 
\[
\frac{\|\mathbf{1}_B\|_{p_n(\cdot)} \|\mathbf{1}_B\|_{p'_n(\cdot)}}{\mu(B)}
\le 4 \bigg(\frac{\|\mathbf{1}_B\|_{p(\cdot)} 
\|\mathbf{1}_B\|_{p'(\cdot)}}{\mu(B)}\bigg)^{1-\theta}
\le 4[p]_\mathcal{A}^{1-\theta}\le8[p]_\mathcal{A}, 
\]
where we used $[p]_\mathcal{A}^{-\theta}\le2$ as a consequence of 
$[p]_\mathcal{A}\ge\frac12$. Taking the supremum over all balls~$B$ proves~(c). 

For~(d), let $p_\infty$ be the Nekvinda parameter for 
$p(\cdot)\in\mathcal{N}(\Omega,\mu)$. Define $(p_n)_\infty$ by 
\[
\frac{1}{(p_n)_\infty}=\frac{1}{n+1}+\frac{n-1}{n+1}\,\frac{1}{p_\infty}.
\]
As directly follows for all $x\in\Omega$,
\[
\frac1{s_n(x)}:=\bigg|\frac1{p_n(x)}-\frac1{(p_n)_\infty}\bigg|\le
\bigg|\frac1{p(x)}-\frac1{p_\infty}\bigg|=:\frac1{s(x)}.
\]
By Theorem~\ref{th:toolbox}(f), the squeeze $s(\cdot)\le s_n(\cdot)\le\infty$ results in
\[
[p_n]_\mathcal{N}:=\|1\|_{s_n(\cdot)}
\le2\,\max\{\|1\|_{s(\cdot)},\|1\|_\infty\}=2[p]_\mathcal{N}.
\]

Finally, we prove~(e). Set $r(\cdot):=\min\{p(\cdot),2\}$ and 
$q(\cdot):=\max\{p(\cdot),2\}$. Since $p_n(x)$ lies between $p(x)$ and $2$ for all
$x\in\Omega$, we have $r(\cdot)\le p_n(\cdot)\le q(\cdot)$. Due to the lower 
semi-continuity from Theorem~\ref{th:toolbox}(k), it suffices to prove that 
\begin{equation}\label{ineq:limsup}
\limsup_{n\to\infty} \|f\|_{\tilde{\varrho}_{p_n(\cdot)}}
\le \|f\|_{\tilde{\varrho}_{p(\cdot)}}.
\end{equation}
Consider the nontrivial case $0<\|f\|_{\tilde{\varrho}_{p(\cdot)}}<\infty$.
Fix $\lambda\in(0,1)$. Splitting $\Omega$ into $\{p(\cdot)\le2\}$ and $\{p(\cdot)>2\}$, 
and applying the unit ball property along with $f\in L^2(\Omega,\mu)$, gives
\[
\tilde{\varrho}_{r(\cdot)}\Big(\frac{\lambda f}{\|f\|_{\tilde{\varrho}_{p(\cdot)}}}\Big)
+\tilde{\varrho}_{q(\cdot)}\Big(\frac{\lambda f}{\|f\|_{\tilde{\varrho}_{p(\cdot)}}}\Big)
=
\tilde{\varrho}_{p(\cdot)}\Big(\frac{\lambda f}{\|f\|_{\tilde{\varrho}_{p(\cdot)}}}\Big)
+\tilde{\varrho}_2\Big(\frac{\lambda f}{\|f\|_{\tilde{\varrho}_{p(\cdot)}}}\Big)
<\infty.
\]
Therefore, Theorem~\ref{th:toolbox}(l) and the convexity of $\tilde{\varrho}_{p(\cdot)}$
imply
\[
\lim_{n\to\infty} \tilde{\varrho}_{p_n(\cdot)}
\Big(\frac{\lambda f}{\|f\|_{\tilde{\varrho}_{p(\cdot)}}}\Big)
=\tilde{\varrho}_{p(\cdot)}
\Big(\frac{\lambda f}{\|f\|_{\tilde{\varrho}_{p(\cdot)}}}\Big)
\le \lambda\, \tilde{\varrho}_{p(\cdot)}
\Big(\frac{f}{\|f\|_{\tilde{\varrho}_{p(\cdot)}}}\Big)\le\lambda <1.
\]
By the unit ball property, 
$\lambda\,\|f\|_{\tilde{\varrho}_{p_n(\cdot)}}
\le \|f\|_{\tilde{\varrho}_{p(\cdot)}}$ for all sufficiently large~$n$. Letting
$\lambda\to1^-$ establishes~\eqref{ineq:limsup} and finishes the proof of~(e).
\end{proof}
Finally, we establish the main theorem for unbounded spaces without 
the restriction $p_+<\infty$. We achieve this by approximating the unbounded exponent 
via Theorem~\ref{th:approximation}. Since the uniform norm bounds in 
Corollary~\ref{cor:aux-US} are independent of $p_+$, we apply them 
to the approximating sequence and safely pass to the limit. The properties of 
the adjacent dyadic grids then transfer this bound from the dyadic 
operators~$M^\mathscr{D}$ to the continuous maximal operator~$M$.
\begin{theorem}[Main theorem in unbounded spaces]\label{th:main-US}
Let $(\Omega,d,\mu)$ be an \emph{unbounded} space of homogeneous type with a 
Borel-semiregular measure~$\mu$. Suppose that 
$p(\cdot)\in\mathcal{A}_\textnormal{loc}(\Omega,d,\mu)\cap\mathcal{N}(\Omega,\mu)$
and $p_->1$. Then for all $f\in L^{p(\cdot)}(\Omega,d,\mu)$,
\[
\|Mf\|_{p(\cdot)}\le c\,\|f\|_{p(\cdot)},
\]
where $c>0$ depends only on $(p')_+$, $[p]_\mathcal{A}$, $[p]_\mathcal{N}$, and the 
structural constants of~$(\Omega,d,\mu)$.
\end{theorem}
\begin{proof}
Fix $x_0\in\Omega$ and define the truncations 
$f_k:=\min\{|f|,k\}\,\mathbf{1}_{B(x_0,k)}$ for $k\in\mathbb{N}$. Since each $f_k$
is a bounded function supported in a ball, $f_k\in L^2(\Omega,d,\mu)$.
Take a dyadic grid~$\mathscr{D}$ on~$\Omega$. The classical $L^2$-boundedness 
of~$M^\mathscr{D}$ ensures $M^\mathscr{D}f_k\in L^2(\Omega,d,\mu)$.

Let $p_n(\cdot)$ be bounded approximations of~$p(\cdot)$ from
Theorem~\ref{th:approximation}. Lemma~\ref{le:2-Mucken} guarantees
$p_n(\cdot)\in\mathcal{A}_\textnormal{loc}^\mathscr{D}(\Omega,d,\mu)$ 
with $[p_n]_\mathcal{A}^\mathscr{D}\le C_\Omega\,[p_n]_\mathcal{A}$, where $C_\Omega>1$ 
depends only on the structural constants of the underlying space. Overall,
\[
p_n(\cdot)\in\mathcal{A}_\textnormal{loc}^\mathscr{D}(\Omega,d,\mu)
\cap\mathcal{N}(\Omega,\mu)
\quad\textnormal{and}\quad
1<(p_n)_-\le(p_n)_+<\infty,
\]
with the uniform upper bounds $(p'_n)_+\le\max\{(p')_+,2\}$,
$[p_n]_\mathcal{A}^\mathscr{D}\le 8C_\Omega\,[p]_\mathcal{A}$, and 
$[p_n]_\mathcal{N}\le2[p]_\mathcal{N}$. Applying Theorem~\ref{th:aux-US} (in view
of Corollary~\ref{cor:aux-US}) and 
norm equivalence~\eqref{eq:norms} yields
\[
\|M^\mathscr{D}f_k\|_{\tilde{\varrho}_{p_n(\cdot)}}\le 
c\,\|f_k\|_{\tilde{\varrho}_{p_n(\cdot)}},
\]
for all $k,n\in\mathbb{N}$, 
where $c$ depends only on $(p')_+$, $[p]_\mathcal{A}$, $[p]_\mathcal{N}$, $C_\Omega$, 
and $C_\mathrm{D}$. 

Letting $n\to\infty$ and using Theorem~\ref{th:approximation}(e)
for both $f_k$ and $M^\mathscr{D}f_k$ gives
\[
\|M^\mathscr{D}f_k\|_{\tilde{\varrho}_{p(\cdot)}}\le 
c\,\|f_k\|_{\tilde{\varrho}_{p(\cdot)}}.
\]
Since $0\le f_k\nearrow|f|$, we also have $M^\mathscr{D}f_k\nearrow M^\mathscr{D}f$.
Letting $k\to\infty$ via the Fatou property from 
Theorem~\ref{th:toolbox}(g) and transitioning to the equivalent norm 
establishes
\begin{equation}\label{eq:bound-dyadic}
\|M^\mathscr{D}f\|_{p(\cdot)}\le 2c\,\|f\|_{p(\cdot)}.
\end{equation}

Crucially, the constant~$2c$ relies only on $(p')_+$, $[p]_\mathcal{A}$, 
$[p]_\mathcal{N}$, and the structural constants of~$(\Omega,d,\mu)$. It is therefore 
uniform across the adjacent grids $\mathscr{D}^1,\ldots,\mathscr{D}^K$ from 
Theorem~\ref{th:adjacent-grids}. Applying the maximal 
bound~\eqref{eq:bound-dyadic} to each of the $K$~grids yields
\[
\|Mf\|_{p(\cdot)}\le C\,\sum_{t=1}^K \|M^{\mathscr{D}^t}f\|_{p(\cdot)}
\le 2c\,CK\,\|f\|_{p(\cdot)},
\]
completing the proof.
\end{proof}
\begin{remark}
As in the Euclidean theory~\cite[Proposition~4.9]{CUF13}, the Nekvinda condition can be 
replaced by the stronger (but easier to verify) requirement of \emph{logarithmic 
decay at infinity}: suppose there exist~$x_0\in\Omega$, $p_\infty\in[1,\infty]$, 
and~$c_\infty>0$ such that
\[
\bigg|\frac{1}{p(x)}-\frac1{p_\infty}\bigg|\le\frac{c_\infty}{\log(e+d(x,x_0))},
\quad x\in\Omega.
\]

This condition implies~$p(\cdot)\in\mathcal{N}(\Omega,\mu)$. Indeed,
let the exponent $s(\cdot)$ be defined by~\eqref{cond:N}. For all integers~$k\ge0$,
define the balls $B_k:=B(x_0,2^k)$. Decompose~$\Omega$ into the base ball~$B_0$ and 
the dyadic annuli $A_k:=B_k\setminus B_{k-1}$ for $k\in\mathbb{N}$. The logarithmic
decay guarantees that for all $x\in A_k$,
\[
s(x)\ge\frac{1}{c_\infty}\,\log(e+2^{k-1})\ge ck,
\]
where $c>0$ is a constant independent of~$k$. Since the doubling 
condition~\eqref{cond:doubling} implies $\mu(A_k)\le C_\mu^k\,\mu(B_0)$,
for all~$\lambda>1$ we estimate
\begin{align*}
\varrho_{s(\cdot)}(1/\lambda)&=\int_{B_0} \lambda^{-s(x)}\,d\mu(x)
+\sum_{k=1}^\infty\int_{A_k} \lambda^{-s(x)}\,d\mu(x)
\\
&\le\mu(B_0)+\sum_{k=1}^\infty \lambda^{-ck}\mu(A_k)
\\
&\le\mu(B_0)\,\bigg[1+\sum_{k=1}^\infty (\lambda^{-c}\,C_\mu)^k\bigg].
\end{align*}
Choosing $\lambda$ sufficiently large makes the series convergent, 
yielding $1\in L^{s(\cdot)}(\Omega,\mu)$.
\end{remark}
For bounded spaces of homogeneous type, the analogue of Theorem~\ref{th:main-US} holds 
entirely without the Nekvinda condition.
\begin{theorem}[Main theorem in bounded spaces]
Let $(\Omega,d,\mu)$ be a \emph{bounded} space of homogeneous type with a 
Borel-semiregular measure~$\mu$. Suppose that 
$p(\cdot)\in\mathcal{A}_\textnormal{loc}(\Omega,d,\mu)$
and $p_->1$. Then for all $f\in L^{p(\cdot)}(\Omega,d,\mu)$,
\[
\|Mf\|_{p(\cdot)}\le c\,\|f\|_{p(\cdot)},
\]
where $c>0$ depends only on $(p')_+$, $[p]_\mathcal{A}$, the 
structural constants of~$(\Omega,d,\mu)$, and the finite measure~$\mu(\Omega)$.
\end{theorem}
\begin{proof}
The argument proceeds exactly as in the proof of Theorem~\ref{th:main-US}. 
Because the space has finite measure, we employ the simplified truncations 
$f_k:=\min\{|f|,k\}$ and substitute Proposition~\ref{prop:aux-BS} in place of 
Theorem~\ref{th:aux-US}.
\end{proof}
Conversely, Lemma~\ref{le:necess-Muck} establishes that the Muckenhoupt 
condition is necessary for the boundedness of the maximal operator. This yields a
complete characterization for bounded spaces, extending~\cite[Corollary~6.2]{KST26} to 
the case $p_+=\infty$.
\begin{corollary}[Criterion in bounded spaces]\label{cor:criterion}
Let $(\Omega,d,\mu)$ be a \emph{bounded} space of homogeneous type with a 
Borel-semiregular measure~$\mu$. Suppose~$p(\cdot)\in\mathcal{P}(\Omega,\mu)$ 
satisfies~$p_->1$. Then~$M$ is bounded on $L^{p(\cdot)}(\Omega,d,\mu)$ 
\emph{if and only if} $p(\cdot)\in\mathcal{A}_\textnormal{loc}(\Omega,d,\mu)$.
\end{corollary}

\section{An Unbounded Space Where Muckenhoupt Suffices}
\label{sec:US-Muck-suff}
In the Euclidean setting, the Muckenhoupt condition 
$p(\cdot)\in\mathcal{A}_\textnormal{loc}(\mathbb{R}^n)$ alone does not ensure the 
boundedness of the maximal operator on~$L^{p(\cdot)}(\mathbb{R}^n)$, even when 
$1<p_-\le p_+<\infty$. Kopaliani constructed the first 
counterexample~\cite[Theorem~1.3]{TK08}, followed by others 
in~\cite[Theorem~5.3.4]{DHHR11} and~\cite[Example~4.51]{CUF13}. 
Given the full characterization for bounded spaces in Corollary~\ref{cor:criterion}, 
this failure stems from the unboundedness of~$\mathbb{R}^n$.

This led to the open question posed in~\cite[Problem~6.4]{KST26}: in an arbitrary 
\emph{unbounded} space of homogeneous type~$(\Omega,d,\mu)$, is the local 
condition $p(\cdot)\in\mathcal{A}_\textnormal{loc}(\Omega,d,\mu)$ strictly insufficient 
for the boundedness of~$M$ on~$L^{p(\cdot)}(\Omega,d,\mu)$? 

We prove that the answer is \emph{no}. We construct an 
unbounded space of homogeneous type where the Muckenhoupt condition 
suffices for the maximal operator's boundedness on~$L^{p(\cdot)}(\Omega,d,\mu)$ 
with $1<p_-\le p_+<\infty$.
\begin{example}\label{ex:Muckenhoupt}
Let $\Omega:=\{0,1,2,\ldots\}$, and define the distance 
$d:\Omega\times\Omega\to[0,\infty)$ by
\[
d(m,n):=
\begin{cases}
0 & \text{if }m=n,
\\
2^{\max\{m,n\}} & \text{if }m\ne n.
\end{cases}
\]
Define the measure~$\mu$ on all subsets of~$\Omega$ by
\[
\mu(\{0\})=1
\quad\text{and}\quad
\mu(\{n\})=2^{n-1}
\text{ for }n\ge1.
\]
Then~$(\Omega,d,\mu)$ is a space of homogeneous type. Specifically: 
\begin{enumerate}
\item[\textnormal{(i)}] The distance~$d$ is an ultrametric, and 
the space~$(\Omega,d)$ is \emph{unbounded}.

\item[\textnormal{(ii)}]  
Every ball is either a singleton or an initial segment~$I_k:=\{0,1,\ldots,k\}$: 
\[
B(n,r)=
\begin{cases}
\{n\} & \text{if }0<r\le 2^n,
\\
I_k &\text{if }2^k<r\le 2^{k+1} \text{ for some }k\ge n.
\end{cases}
\]

\item[\textnormal{(iii)}] The topology is discrete; hence, every subset of~$\Omega$ 
is Borel, and the measure~$\mu$ is Borel-semiregular. 

\item[\textnormal{(iv)}] The measure~$\mu$ is doubling with constant $C_\mu=2$.

\item[\textnormal{(v)}] The dyadic grid 
$\mathscr{D}=\bigcup_{k\in\mathbb{Z}}\mathscr{D}_k$, defined by
\[
\mathscr{D}_{-\ell}:=\{I_\ell\}\cup\{\{n\}:n>\ell\},
\quad\ell\ge0,
\]
and
\[
\mathscr{D}_k:=\{\{n\}:n\in\Omega\},
\quad k\ge1,
\]
satisfies properties~(a)--(e) of the Hyt\"onen--Kairema construction 
in Theorem~\ref{th:HK} with~$C_D=2$, $R=2$, and $\delta=\frac12$. 

\item[\textnormal{(vi)}] Since every ball is a cube in~$\mathscr{D}$ and vice versa,
the single grid~$\mathscr{D}$ fulfills the properties of the 
adjacent dyadic grids from Theorem~\ref{th:adjacent-grids}
\emph{without the need for multiple shifted grids}.
\end{enumerate}
\end{example}
\begin{proof}
The ultrametric inequality
\[
d(m,n)\le\max\{d(m,k),d(k,n)\}
\quad\text{for all }k,m,n\in\Omega
\]
follows immediately from the definition of~$d$. Since $d(0,n)=2^n\to\infty$ as 
$n\to\infty$, the space~$(\Omega,d)$ is unbounded, establishing~(i). 
The classification of balls in~(ii) is obtained by direct calculation. 
Because each singleton $\{n\}=B(n,2^n)$ is open, the 
topology induced by~$d$ is discrete, yielding~(iii).

To prove the doubling property~(iv), consider a ball~$B(n,r)$. If $0<r\le2^n$, then 
$B(n,r)=\{n\}$, and doubling the radius implies that $B(n,2r)$ is either~$\{n\}$ 
or~$I_n$. By definition,
\[
\mu(I_n):=\mu(\{0\})+\sum_{k=1}^n \mu(\{k\})=1+\sum_{k=1}^n 2^{k-1}=2^n.
\]
Thus, for $n\ge1$,
\[
\frac{\mu(B(n,2r))}{\mu(B(n,r))}\le\frac{\mu(I_n)}{\mu(\{n\})}
=\frac{2^n}{2^{n-1}}=2,
\]
while for $n=0$ the ratio is exactly~$1$. Alternatively, if $2^k<r\le2^{k+1}$ for 
some~$k\ge n$, then 
$B(n,r)=I_k$ and $B(n,2r)=I_{k+1}$. Hence 
\[
\mu(B(n,2r))=2^{k+1}=2\,\mu(B(n,r)). 
\]
This confirms the measure~$\mu$ is doubling with constant~$C_\mu=2$.

For~(v), we verify the Hyt\"onen--Kairema axioms from Theorem~\ref{th:HK}.
By construction, each~$\mathscr{D}_k$ is a partition of~$\Omega$, and all 
non-singleton cubes form the nested chain $I_0\subset I_1\subset I_2\subset\ldots$. 
A singleton either lies in an~$I_k$ or is disjoint from it. The only genuine 
split occurs when $I_k=I_{k-1}\cup\{k\}$ for $k\ge1$. Here, the measure splits 
perfectly in half:
\[
\mu(I_{k-1})=\mu(\{k\})=2^{k-1}=\frac{\mu(I_k)}{2}.
\]
For repeated singletons, the child and parent have equal measure. Therefore, 
the dyadic doubling constant is~$C_\mathrm{D}=2$. To verify the inner and outer ball 
condition~(e), set $\delta=\frac12$ and $R=2$. For generations 
$-\ell\le0$: if $Q=I_\ell\in\mathscr{D}_{-\ell}$, choose $x_Q=0$. Then indeed
\[
B(0,2^\ell)\subseteq I_\ell=B(0,2^{\ell+1}).
\]
If $Q=\{n\}\in\mathscr{D}_{-\ell}$ with $n>\ell$, choose $x_Q=n$. Since 
$2^{\ell+1}\le2^n$, we have
\[
B(n,2^\ell)=\{n\}=B(n,2^{\ell+1}).
\]
For generations~$k\ge1$, every cube is a singleton. For $Q=\{n\}\in\mathscr{D}_k$,
choosing $x_Q=n$ trivially yields
\[
B(n,2^{-k})=\{n\}=B(n,2^{-k+1}).
\]
Thus, $\mathscr{D}$ is an admissible Hyt\"onen--Kairema dyadic grid.

Finally, claim~(vi) follows by noting that the operators~$M$ and $M^\mathscr{D}$
coincide.
\end{proof}
The sufficiency of the Muckenhoupt condition in this space follows from our recent 
characterization of the maximal operator's boundedness~\cite{AS26}.
\begin{proposition}
Let~$(\Omega,d,\mu)$ be the \emph{unbounded} space of homogeneous type defined in 
Example~\ref{ex:Muckenhoupt}. Suppose that 
$p(\cdot)\in\mathcal{A}_\textnormal{loc}(\Omega,d,\mu)$ and $1<p_-\le p_+<\infty$. 
Then~$M$ is bounded on~$L^{p(\cdot)}(\Omega,d,\mu)$.
\end{proposition}
\begin{proof}
Consider the dyadic grid~$\mathscr{D}$ from Example~\ref{ex:Muckenhoupt}(v).
For a function $f\in L^1_\textnormal{loc}(\Omega,d,\mu)$ and a cube~$Q\in\mathscr{D}$, 
denote $A_Qf:=|f|_Q\mathbf{1}_Q$. Because every cube in~$\mathscr{D}$ is a ball, 
H\"older's inequality and the Muckenhoupt condition yield
\begin{equation}\label{ineq:AQ}
\|A_Qf\|_{p(\cdot)}=|f|_Q\,\|\mathbf{1}_Q\|_{p(\cdot)}
\le\frac{2\,\|f\|_{p(\cdot)}}{\mu(Q)}\,\|\mathbf{1}_Q\|_{p'(\cdot)}
\|\mathbf{1}_Q\|_{p(\cdot)}
\le2[p]_\mathcal{A}\,\|f\|_{p(\cdot)}.
\end{equation}

Fix a family $\mathcal{Q}$ of pairwise disjoint cubes from~$\mathscr{D}$, and 
define
\[
T_\mathcal{Q}f:=\sum_{Q\in\mathcal{Q}} A_Qf.
\]
Due to the nested geometry of the space, $\mathcal{Q}$ can contain at most one 
non-singleton cube. 
Let $S:=\{n:\{n\}\in\mathcal{Q}\}$. If~$\mathcal{Q}$ consists only of singletons,
then
\[
\|T_\mathcal{Q}f\|_{p(\cdot)}=
\bigg\|\sum_{n\in S} |f(n)|\,\mathbf{1}_{\{n\}}\bigg\|_{p(\cdot)}
=\|\,|f|\mathbf{1}_S\|_{p(\cdot)}\le\|f\|_{p(\cdot)}.
\]
If, alternatively, $\mathcal{Q}$ also contains an initial segment~$I_\ell$, 
then~\eqref{ineq:AQ} yields 
\[
\|T_\mathcal{Q}f\|_{p(\cdot)}\le
\|A_{I_\ell}f\|_{p(\cdot)}+\|\,|f|\mathbf{1}_S\|_{p(\cdot)}
\le (2[p]_\mathcal{A}+1)\,\|f\|_{p(\cdot)}.
\]
Consequently, the averaging operators~$T_\mathcal{Q}$ are bounded on 
$L^{p(\cdot)}(\Omega,d,\mu)$ uniformly over all families~$\mathcal{Q}$ of pairwise 
disjoint cubes from~$\mathscr{D}$.

To conclude, we invoke the characterization from~\cite[Theorem~1.2]{AS26}: 
on an unbounded space of homogeneous type with a Borel-semiregular measure, 
given $1<p_-\le p_+<\infty$, the maximal operator is bounded
on~$L^{p(\cdot)}(\Omega,d,\mu)$ if and only if the operators~$T_\mathcal{Q}$ are 
bounded on~$L^{p(\cdot)}(\Omega,d,\mu)$ uniformly over all families~$\mathcal{Q}$ of 
pairwise disjoint cubes \emph{from the union of the Hyt\"onen--Kairema adjacent 
dyadic grids}. 

Since every ball is a cube in~$\mathscr{D}$, our single grid already provides the 
requisite fine covering property of the adjacent grids~\cite[Theorem~3.4]{AS26}.
Thus, the established uniform bound on~$T_\mathcal{Q}$ immediately yields
the boundedness of~$M$ on~$L^{p(\cdot)}(\Omega,d,\mu)$.
\end{proof}

\section*{Declarations}
\textbf{Conflict of interest:} The author declares no conflicts of interest.

\vspace{0.5em}\noindent
\textbf{Data availability:} Data sharing not applicable to this article as no datasets 
were generated or analysed during the current study.

\bibliographystyle{abbrv}
\bibliography{AS-bib}

\end{document}